\documentclass[12pt]{amsart}
\usepackage{amsmath,amssymb,amsthm,amsfonts}
\usepackage{cite}
\usepackage{fullpage}
\usepackage[colorlinks=true]{hyperref}
\usepackage{mathrsfs}

\newtheorem{theorem}{Theorem}[section]
\newtheorem{proposition}[theorem]{Proposition}
\newtheorem{lemma}[theorem]{Lemma}
\newtheorem{corollary}[theorem]{Corollary}
\theoremstyle{remark}
\newtheorem{remark}{Remark}[section]

\theoremstyle{definition}
\newtheorem{definition}[theorem]{Definition}

\allowdisplaybreaks
\numberwithin{equation}{section}

\newcommand{\R}{\mathbb{R}}

\keywords{Muskat problem, surface tension, self-similar solutions}
\subjclass[2020]{35R35, 35Q35, 35A01, 35A02}
\author{Lizhe Wan}
\address{Beijing International Center for Mathematical Research, Peking University}
\curraddr{}
\email{wanlizhe@pku.edu.cn}

\author{Jiaqi Yang}
\address{School of Mathematics and Statistics, Northwestern Polytechnical University}
\curraddr{}
\email{yjqmath@nwpu.edu.cn, yjqmath@163.com}

\begin{document}

\title{Self-similar solutions of the three-dimensional Muskat problem with surface tension}

\begin{abstract}
We construct a one-parameter family of small self-similar solutions
to the three-dimensional one-phase Muskat problem with surface
tension. The solutions have the form
\(\eta_\varepsilon(t,x)
=t^{1/3}U_\varepsilon(t^{-1/3}x)\)
and emanate from the conical initial data
\(\eta_\varepsilon(0,x)=\varepsilon|x|\).
The profiles are perturbations of the linear capillary regularization of
the cone, and we identify the leading quadratic correction. The proof
combines a raywise inverse estimate for the linear
similarity operator, a favorable high--high-to-low cancellation in
the quadratic term, and finite-order tame estimates for the
Dirichlet--Neumann operator on asymptotically conical graphs.
These estimates yield the solutions by a contraction argument and
show that the conical singularity is instantaneously rounded for
positive time.
\end{abstract}

\maketitle
%\tableofcontents

\section{Introduction}
We consider a three-dimensional incompressible fluid in a porous medium,
with a free upper boundary subject to surface tension. Our aim is to
construct self-similar solutions of the corresponding one-phase Muskat
problem for interfaces over \(\R^2\).
Let \(\Sigma_t\) denote the interface between the fluid and the air at time \(t\).
Throughout the paper, we assume that \(\Sigma_t\) is the graph of a function \(\eta(t,x)\):
\begin{equation*}
 \Sigma_t=\{(x,\eta(t,x)):x\in\R^2\}.
\end{equation*}
The fluid domain is then given by
\begin{equation*}
 \Omega_t=\{(x,y)\in\R^2\times\R:y<\eta(t,x)\}.
\end{equation*}
We normalize the permeability of the porous medium to one and take the constant pressure in the air phase to be zero.

The velocity \(u\) and pressure \(p\) satisfy Darcy's law and the incompressibility condition
\begin{equation*}
 \mu u+\nabla_{x,y}p=0, \qquad \operatorname{div}_{x,y}u=0 \quad\text{in }\Omega_t,
\end{equation*}
where \(\mu>0\) is the viscosity of the fluid.

Let
\begin{equation*}
 n=\frac{1}{\sqrt{1+|\nabla\eta|^2}}(-\nabla\eta,1)
\end{equation*}
denote the upward-pointing unit normal along \(\Sigma_t\).
The kinematic boundary condition is
\begin{equation*}
 \partial_t\eta=\sqrt{1+|\nabla\eta|^2}\,u\cdot n\big|_{\Sigma_t}.
\end{equation*}
The dynamic boundary condition states that the pressure jump is balanced by surface tension:
\begin{equation*}
 p=\sigma\kappa(\eta) \quad\text{on }\Sigma_t, \qquad \kappa(\eta)=-\operatorname{div}\left(\frac{\nabla\eta}{\sqrt{1+|\nabla\eta|^2}}\right),
\end{equation*}
where \(\sigma>0\) is the surface-tension coefficient.

\subsection{Formulation and self-similar scaling of the Muskat problem}
The Muskat problem admits both contour-dynamics and
Dirichlet--Neumann formulations; see, for instance,
\cite{MR4520423,MR3415681,MR3171344,MR2313156,MR2993754,
MR3608884,MR3071395}.

Following Nguyen~\cite{MR4131404} and Nguyen and Pausader~\cite{MR4090462},
we use the Dirichlet--Neumann operator \(G(\eta)\) associated with
\(\Omega_t\). For a Lipschitz graph and
\(f\in\dot H^{1/2}(\R^2)\), regarded as boundary data through the graph
parametrization, let \(\phi_f\) denote the energy harmonic extension. It is
unique at the level of its gradient and is characterized by
\begin{equation*}
 \Delta_{x,y}\phi_f=0 \quad\text{in }\Omega_t, \qquad
 \phi_f=f \quad\text{on }\Sigma_t, \qquad
 \nabla_{x,y}\phi_f\in L^2(\Omega_t).
\end{equation*}
For smooth \(\eta\) and \(f\), we define the Dirichlet--Neumann operator by
\begin{equation*}
 G(\eta)f:=\sqrt{1+|\nabla\eta|^2}\,\partial_n\phi_f\big|_{\Sigma_t}.
\end{equation*}
For the rough data used below, the same notation denotes the weak conormal trace of the energy extension.

The following reduction is proved in Appendix~B of
\cite{MR4131404}. It applies either to sufficiently regular solutions or,
as in our construction, in the energy graph formulation whenever the
relevant traces are defined.

\begin{proposition}[\hspace{1sp}\cite{MR4131404}]
If \((u,p,\eta)\) is a solution of the one-phase Muskat problem in the preceding class, then \(\eta\) satisfies
\begin{equation} \label{OneMuskat}
 \partial_t\eta=-\frac{\sigma}{\mu}G(\eta)\kappa(\eta).
\end{equation}
Conversely, a solution \(\eta\) of \eqref{OneMuskat} for which the energy Dirichlet--Neumann operator is defined determines a solution of the one-phase Muskat problem whose free surface is parametrized by \(\eta\).
\end{proposition}

Equation \eqref{OneMuskat} is a third-order quasilinear parabolic equation.
At the linearized level, \(G(\eta)\) is replaced by \(|D|\) and
\(\kappa(\eta)\) by \(|D|^2\eta\), so that
\begin{equation*}
 \partial_t\eta=-\frac{\sigma}{\mu}|D|^3\eta+\text{nonlinear terms}.
\end{equation*}

If \(\eta(t,x)\) solves \eqref{OneMuskat}, then so does
\begin{equation*}
 \eta_\lambda(t,x)=\lambda^{-1}\eta(\lambda^3t,\lambda x), \qquad \lambda>0.
\end{equation*}

This scaling leads to the following notion of self-similarity.
\begin{definition}
A solution \(\eta(t,x)\) of \eqref{OneMuskat} is called self-similar if it can be written as
\begin{equation*}
 \eta(t,x)=t^{1/3}U\left(t^{-1/3}x\right), \qquad t>0,
\end{equation*}
for some time-independent profile \(U\) on \(\R^2\).
Whenever the following limit exists locally uniformly in \(x\), we define the initial trace by
\begin{equation*}
 \eta(0,x)=\lim_{\rho\to0^+}\rho U\left(\rho^{-1}x\right).
\end{equation*}
\end{definition}

Substitution of the self-similar ansatz into \eqref{OneMuskat} gives the profile equation
\begin{equation*}
 U-x\cdot\nabla U+\frac{3\sigma}{\mu}G(U)\kappa(U)=0.
\end{equation*}
Zero solution is obviously a trivial self-similar solution, it therefore suffices to find a non-trivial solution to the profile equation.  
Writing the profile equation as
\begin{equation} \label{AbstractEqn}
 \mathcal LU=\frac{3\sigma}{\mu}\mathscr F(U),
\end{equation}
where
\begin{equation} \label{LFDef}
 \mathcal L:=\frac{3\sigma}{\mu}|D|^3-x\cdot\nabla+1, \qquad \mathscr F(U):=|D|^3U-G(U)\kappa(U),
\end{equation}
it remains to find nontrivial solutions of \eqref{AbstractEqn}.

\subsection{Related work and the main result}
\label{s:related-work-main-result}

The Cauchy theory for the Muskat and Hele--Shaw problems has been
extensively developed in subcritical and critical Sobolev, Lipschitz, and
Besov spaces.
For the gravity-driven problem, local well-posedness and global existence under suitable stability or smallness assumptions can be found in
\cite{MR2318314,MR2753607,MR3071395,MR3415681,MR3861893,
MR4097324,MR4090462,MR4242131,MR4313450,MR4387237,
MR4402655,MR4541917,MR4986619}
and the references therein.

In the presence of surface tension, the leading-order operator is the third-order dissipative operator \(|D|^3\). 
Large-data local well-posedness in subcritical Sobolev spaces and further critical or global results were established in \cite{MR4131404,MR2162781,MR3841857,CHN,Lazar}.
Dong and Kwon~\cite{DongKwon26} proved global well-posedness for the one-phase Muskat problem with surface tension for sufficiently small initial data in \(H^s\), \(s>d/2+1\), where \(d\) denotes the interface dimension, together with decay of the Lipschitz norm.

These Cauchy theories do not by themselves produce the self-similar solutions considered here. 
Even when Lipschitz initial interfaces are allowed, the problem of constructing a solution with the exact unbounded initial trace
\[
 \eta_0(x)=\varepsilon|x|
\]
requires preserving both the self-similar scaling and the degree-one conical behavior at infinity.

For the gravity-driven Muskat problem, Garc{\'\i}a-Ju{\'a}rez,
G{\'o}mez-Serrano, Nguyen, and Pausader~\cite{GGNP22} constructed
planar self-similar solutions emanating from exact corners. Na
\cite{Na25} subsequently constructed radially symmetric
self-similar solutions for the three-dimensional problem with exact
cones as initial data and compared them with the corresponding
linearized profiles in critical homogeneous Sobolev and weighted
\(\dot W^{k,\infty}\)-type spaces. In these gravity-driven problems,
the principal operator is of first order, and the nonlinear
correction in Na's formulation starts at cubic order.
For the two-dimensional one-phase Hele--Shaw problem with surface tension, Agrawal and Patel~\cite{agrawal26} constructed planar self-similar solutions emanating from corners with opening angle close to \(\pi\). 

We treat the surface-tension problem directly in physical graph
coordinates for an interface over \(\R^2\). The main result is a
one-parameter family of self-similar solutions whose initial traces are
small multiples of the cone \(|x|\).

\begin{theorem}[Existence of self-similar solutions]
\label{t:MainTheorem}
Let \(\sigma,\mu>0\) and \(3/2<r<2\).
Then there exists \(\varepsilon_0>0\) such that, for every \(\varepsilon\in\R\) with \(|\varepsilon|\leq\varepsilon_0\), the one-phase Muskat equation with surface tension \eqref{OneMuskat} admits a self-similar solution in the energy graph formulation,
\begin{equation*}
 \eta_\varepsilon(t,x)=t^{1/3}U_\varepsilon(t^{-1/3}x), \qquad t>0,
\end{equation*}
whose profile is of the form
\begin{equation*}
 U_\varepsilon=\varepsilon u_*+v_\varepsilon, \qquad u_*=e^{-\frac{\sigma}{\mu}|D|^3}|x|.
\end{equation*}
Here,
\begin{equation*}
  v_\varepsilon\in X_r:=\dot H^r(\R^2)\cap\dot H^{r+3}(\R^2), \qquad \|v_\varepsilon\|_{X_r}\leq C_{r,\sigma/\mu}\varepsilon^2.
\end{equation*}
The correction \(v_\varepsilon\) is the normalized representative selected by the canonical inverse \(\mathcal L^{-1}\) constructed in Proposition~\ref{prop:linear-inverse-Sobolev}.
With this normalization, the profile equation holds in \(\mathcal S'(\R^2)\), and
\begin{equation*}
 \nabla U_\varepsilon,D^2U_\varepsilon\in L^\infty(\R^2).
\end{equation*}
Define the quadratic correction 
\begin{equation*}
 v^{(2)}:=\frac{3\sigma}{\mu}\mathcal L^{-1}\mathcal B(u_*,u_*), \qquad  \mathcal B(f,g) := \frac12 \left( D_\eta G(0)[f]\Delta g + D_\eta G(0)[g]\Delta f \right).
\end{equation*}
Then \(v^{(2)}\in X_r\) and
\begin{equation*}
 \left\|v_\varepsilon-\varepsilon^2v^{(2)}\right\|_{X_r}\leq C_{r,\sigma/\mu}|\varepsilon|^3.
\end{equation*}
Equivalently,
\[
 U_\varepsilon=\varepsilon u_*+\varepsilon^2v^{(2)}+O_{X_r}(|\varepsilon|^3) \qquad\text{as }\varepsilon\to0.
\]
Its initial trace is attained locally uniformly:
\[
 \eta_\varepsilon(t,x)\longrightarrow\varepsilon|x|
 \qquad\text{locally uniformly in }x\text{ as }t\downarrow0.
\]
Thus \(\eta_\varepsilon(0,x)=\varepsilon|x|\).
In particular, the solution is nontrivial whenever \(\varepsilon\neq0\).
\end{theorem}

\begin{remark}[Meaning of the quadratic correction]
The leading term \(\varepsilon u_*\) is the regularized linear cone, whereas \(v^{(2)}\) describes the first nonlinear deformation of its self-similar core.
Since \(v^{(2)}\in X_r\) and \(r<2\), the homogeneous
Sobolev--Morrey inequality in dimension two gives
\[
 |v^{(2)}(x)-v^{(2)}(0)|
 \leq
 C_r|x|^{r-1}\|v^{(2)}\|_{\dot H^r}.
\]
Thus, for every \(R>0\),
\begin{align*}
\sup_{|x|\leq R}
 t^{1/3}
 \left|
  v^{(2)}(t^{-1/3}x)
 \right|\leq
 t^{1/3}|v^{(2)}(0)|
 +
 C_rR^{r-1}
 t^{(2-r)/3}
 \|v^{(2)}\|_{\dot H^r}.
\end{align*}
Since \(r<2\), the right-hand side tends to zero as \(t\downarrow0\), and
\[
 t^{1/3}v^{(2)}(t^{-1/3}x)\longrightarrow0 \qquad\text{locally uniformly as }t\downarrow0.
\]
Thus the quadratic correction modifies the rounded core for \(t>0\) but does not change the conical initial trace \(\varepsilon|x|\).
\end{remark}

\begin{remark}[The one-dimensional case]
\label{rem:one-dimensional-case}
When \(d=1\), the problem reduces to the planar one-phase Hele--Shaw problem with surface tension. 
Agrawal and Patel~\cite{agrawal26} constructed self-similar solutions emanating from corners with opening angle close to \(\pi\).
The initial graph \(y=\varepsilon|x|\) is a wedge of this type. 
With the fluid lying below the graph, its opening angle is
\[
 \pi+2\arctan\varepsilon. 
\]
Thus the existence of planar self-similar solutions with this type
of initial wedge is already covered by Agrawal and
Patel~\cite{agrawal26}, although the graph-coordinate approach
described here would provide a different proof.

Our argument is expected to adapt to \(d=1\).  A natural choice of
correction space would be
\[
 1<r<\frac32,\qquad
 X_r^{(1)} := \dot H^r(\R)\cap\dot H^{r+3}(\R).
\]
The quadratic cancellation becomes particularly transparent: writing
\(\zeta=\alpha+\beta\),
\[
 |\zeta||\beta|-\zeta\beta = 2|\zeta||\beta|
 \mathbf 1_{\{\zeta\beta<0\}}.
\]
Hence interactions for which \(\zeta\beta\geq0\) vanish identically, while the high--high-to-low interactions retain the required output-frequency gain. A rigorous one-dimensional extension would additionally
require the dimension-adjusted linear Hardy estimate and mixed
Dirichlet--Neumann estimates.  We do not pursue these details here.
\end{remark}

\begin{remark}[Higher dimensions]
\label{rem:higher-dim}
The algebraic argument used in this paper is independent of the dimension.
For a graph over \(\R^d\), the distinguished linear profile is
\[
 u_{*,d} := e^{-\frac{\sigma}{\mu}|D|^3}|x|, \qquad
 \widehat{u_{*,d}}(\xi)
 = c_d e^{-\frac{\sigma}{\mu}|\xi|^3} |\xi|^{-d-1}
\]
in \(\mathcal S'(\R^d)/\mathcal P\). 
The same raywise computation
shows that \(u_{*,d}\) belongs to the kernel of the corresponding linear similarity operator.
Moreover, the symbolic identities proved in Lemma~\ref{lem:quadratic-cancellation-revised},
\[
0 \leq |\alpha+\beta||\beta| - (\alpha+\beta)\cdot\beta \leq \frac12|\alpha|^2, \qquad
 |b(\alpha,\beta)| \lesssim |\alpha+\beta||\beta|^3
\]
show that the favorable quadratic cancellation persists in every dimension.

What depends on the dimension is the choice of the auxiliary
homogeneous Sobolev norms used to control the low frequencies.
Indeed,
\[
 \Delta u_{*,d}\in\dot H^{s_0}(\R^d)
 \quad\Longleftrightarrow\quad
 s_0>\frac d2-1.
\]
Therefore, the two-dimensional low-frequency norms used below for
\(\Delta h\) must be modified when \(d\geq3\).  A natural analogue of the two-dimensional choice is
\[
 s_0=\frac{d-1}{2},
\]
together with the two low-frequency norms \(\dot H^{s_0}\) and \(\dot H^{s_0+1/2}\).
On the other hand, the linear Hardy estimate and the homogeneous Morrey bound suggest the upper restriction
\[
 r<\frac d2+1,
\]
whereas the mixed Sobolev embeddings require \(r>d/2\).
For instance, one may work in the more restrictive range
\[
 \frac{d+1}{2}<r<\frac d2+1,
\]
which is the direct analogue of the range used in dimension two.
A complete higher-dimensional argument would also require corresponding modifications of the mixed half-space estimates.
\end{remark}

\begin{remark}[The initial profile]
\label{rem:initial-profile}
The distinguished profile \(u_*\) is the time-one regularization of the homogeneous cone \(|x|\) under the semigroup generated by \(-\frac{\sigma}{\mu}|D|^3\).
More precisely,
\[
 t^{1/3}u_*(t^{-1/3}x) = e^{-\frac{\sigma}{\mu}t|D|^3}|x|.
\]
The convergence back to the cone can be checked directly despite the
unboundedness of \(|x|\).
With the normalization used throughout the paper,
\[
 \widehat{u_*-|x|}(\xi) = c_2 \bigl( e^{-\frac{\sigma}{\mu}|\xi|^3}-1 \bigr)|\xi|^{-3}.
\]
The right-hand side belongs to \(L^1(\R^2)\): it is bounded near \(\xi = \mathbf{0}\) and is \(O(|\xi|^{-3})\) as \(|\xi|\to\infty\).
Consequently,
\[
 u_*(x) = |x|+\psi_{\sigma/\mu}(x), \qquad \psi_{\sigma/\mu}\in C_0(\R^2)\cap L^\infty(\R^2).
\]
It follows that
\[
 t^{1/3}u_*(t^{-1/3}x) = |x|+t^{1/3}\psi_{\sigma/\mu}(t^{-1/3}x) \longrightarrow |x|
\]
locally uniformly in \(x\) as \(t\downarrow0\).

For the normalized representative of \(v_\varepsilon\), the homogeneous
Morrey estimate and the condition \(r<2\) similarly imply
\[
 t^{1/3}v_\varepsilon(t^{-1/3}x) \longrightarrow0 \qquad \text{as }t\downarrow0.
\]
Consequently,
\[
 \eta_\varepsilon(t,x) = t^{1/3} U_\varepsilon(t^{-1/3}x) \longrightarrow \varepsilon|x|
\]
locally uniformly in \(x\), giving the initial trace stated in Theorem~\ref{t:MainTheorem}.

For \(\varepsilon\neq0\), the initial interface has a conical singularity at the origin.
For every \(t>0\), the self-similar profile is \(C^1\), so the conical point is smoothed instantaneously.
The singular cone reappears only as \(t\downarrow0\), while the transition region near its vertex has spatial scale \(O(t^{1/3})\).
\end{remark}

\subsection{Main difficulties and strategy of the proof}
\label{s:main-difficulties-strategy}

Three features place the problem outside the direct scope of the standard small-data theory: 
the initial interface is unbounded and asymptotically conical, the capillary operator is of third order, and the profile equation contains the non-translation-invariant dilation term \(x\cdot\nabla\).

The homogeneous cone \(|x|\) does not belong to the correction space, and neither the initial interface nor the self-similar profile is bounded. 
We therefore separate the conical component by writing
\[
 U=\varepsilon u_*+v,
 \qquad
 u_*:=e^{-\frac{\sigma}{\mu}|D|^3}|x|.
\]
The profile \(u_*\) is regular near the vertex, has the same degree-one behavior at infinity as \(|x|\), and satisfies
\[
 \mathcal Lu_*=0, \qquad \mathcal L = \frac{3\sigma}{\mu}|D|^3-x\cdot\nabla+1.
\]
The correction is sought in
\[
 X_r := \dot H^r(\R^2)\cap\dot H^{r+3}(\R^2), \qquad \frac32<r<2.
\]
The condition \(r>3/2\) is required by the mixed Sobolev embedding used in Section~\ref{s:finite-order-DN}.  
Two distinct regularity levels enter the Dirichlet--Neumann estimates.
The low-frequency energy of the boundary datum \(\Delta h\) is controlled by \(\|h\|_{\dot H^{5/2}}\), while the coefficients of the regularized flattening, which involve one derivative of \(h\), are controlled by \(\|h\|_{\dot H^{r+1}}\);
see Lemma~\ref{lem:regularized-extension-two-level}.  
The stronger norm \(\|h\|_{\dot H^{r+3}}\) is instead used to control the high-frequency boundary norm \(\|\Delta h\|_{\dot H^{r+1}}\).
The upper restriction \(r<2\) is used in the weighted Hardy estimate \eqref{eq:weighted-Hardy-linear} for the linear inverse and in the homogeneous Morrey estimate \eqref{eq:normalized-Morrey-linear}.
Indeed, for a normalized \(v\in X_r\),
\[
 t^{1/3}
 \bigl|v(t^{-1/3}x)-v(0)\bigr|
 \lesssim_r
 t^{(2-r)/3}|x|^{r-1}\|v\|_{\dot H^r},
\]
which tends to zero locally uniformly as \(t\downarrow0\). This is
used in Remark~\ref{rem:initial-profile} to show that the correction
does not alter the conical initial trace.

Although \(\mathcal L\) is not a Fourier multiplier, its Fourier transform reduces along each ray \(\xi=\rho\omega\) to a first-order ordinary differential equation. 
A weighted Hardy estimate applied to the resulting raywise formula gives
\[
 \|\mathcal L^{-1}F\|_{\dot H^r} +
 \|\mathcal L^{-1}F\|_{\dot H^{r+3}}
 \lesssim_{r,\sigma/\mu}  \|F\|_{\dot H^r}.
\]
Thus \(\mathcal L^{-1}\) gains three derivatives at high frequencies
without destroying the conical low-frequency structure.

The principal nonlinear issue is to control \(G(\eta)\) when \(\eta\) is
unbounded but has small slope.
We flatten the graph by the regularized change of variables
\[
 (x,z)\longmapsto \bigl(x,z+H_\eta(x,z)\bigr),
 \qquad H_\eta(x,z):=\chi_{-z}*\eta(x).
\]
The coefficients of the flattened equation depend on derivatives of \(\eta\), rather than on \(\eta\) itself. 
To handle the low frequencies, we introduce mixed half-space norms in which the energy components are unweighted, while high Sobolev weights are imposed only on positive horizontal frequencies. 
For the graph and boundary-data norms \(\mathcal A_r\) and \(\mathcal D_r\) introduced below, there exists \(\delta_r>0\) such that, whenever \(\mathcal A_r(\eta)\leq\delta_r\),
\[
 \left\| D_\eta^kG(\eta)[h_1,\ldots,h_k]f \right\|_{\dot H^r} \lesssim_{r,k} \mathcal D_r(f)
 \prod_{\nu=1}^k\mathcal A_r(h_\nu),
 \qquad 0\leq k\leq3,
\]
uniformly on the small graph ball.

The curvature remainder is cubic:
\[
 \kappa(h) = -\Delta h+\mathcal Q_\kappa(h),
 \qquad \mathcal Q_\kappa(h) = O\bigl((\nabla h)^2D^2h\bigr).
\]
The quadratic part of the nonlinearity is therefore generated by the first shape derivative of \(G\) and is represented by \(\mathcal B(h,h)\). Its symbol satisfies both a global two-derivative cancellation and an additional output-frequency gain in the high--high-to-low region. 
This extra gain makes the homogeneous low-frequency summation convergent.

Combining the quadratic estimate with the finite-order Dirichlet--Neumann calculus yields
\[
 |D|^3h-G(h)\kappa(h) = \mathcal B(h,h)+\mathscr R_{\geq3}(h).
\]
For the conical perturbations
\[
 h=\varepsilon u_*+v, \qquad v\in X_r,
\]
there exists \(\delta_{r,\sigma/\mu}>0\) such that, whenever
\[
 |\varepsilon|+\|v\|_{X_r}\leq\delta_{r,\sigma/\mu},
\]
the remainder satisfies
\begin{equation}
\label{eq:intro-cubic-remainder}
 \|\mathscr R_{\geq3}(\varepsilon u_*+v)\|_{\dot H^r}
 \lesssim_{r,\sigma/\mu} \bigl( |\varepsilon|+\|v\|_{X_r}
 \bigr)^3.
\end{equation}
Moreover, if
\[
 h_i=\varepsilon u_*+v_i,
 \qquad i=1,2,
\]
and \(|\varepsilon|+\|v_i\|_{X_r}\leq\delta_{r,\sigma/\mu}\) for \(i=1,2\),
then
\begin{align}
\label{eq:intro-cubic-remainder-difference}
\|\mathscr R_{\geq3}(h_1)
       -\mathscr R_{\geq3}(h_2)\|_{\dot H^r}\lesssim_{r,\sigma/\mu}
 \bigl(
  |\varepsilon|
  +\|v_1\|_{X_r}
  +\|v_2\|_{X_r}
 \bigr)^2
 \|v_1-v_2\|_{X_r}.
\end{align}

With these estimates in hand, we solve the correction equation by applying
\(\mathcal L^{-1}\) to the quadratic term and the cubic remainder.
The preceding estimates show that the resulting map is a contraction
on a ball of radius \(O(\varepsilon^2)\) in \(X_r\). This gives
\[
 U_\varepsilon
 =
 \varepsilon u_*+v_\varepsilon,
 \qquad
 \|v_\varepsilon\|_{X_r}
 \lesssim_{r,\sigma/\mu}
 \varepsilon^2.
\]
Retaining the leading quadratic contribution in the fixed-point equation
also gives
\[
 v_\varepsilon
 =
 \varepsilon^2v^{(2)}
 +
 O_{X_r}(|\varepsilon|^3),
\]
which is the second-order expansion stated in
Theorem~\ref{t:MainTheorem}.

The remainder of the paper is organized as follows.
In Section~\ref{s:linear-similarity-operator}, we analyze the linear
similarity operator and construct the regularized conical profile.
Section~\ref{s:cone-capillary-expansion-revised} introduces the
cone-correction space and identifies the favorable quadratic cancellation in
the capillary nonlinearity.
In Section~\ref{s:finite-order-DN}, we establish the finite-order
cone-adapted Dirichlet--Neumann estimates.
Finally, Section~\ref{s:capillary-fixed-point} controls the cubic
remainder, closes the fixed-point argument, and recovers the
corresponding self-similar Muskat solution.

\textbf{Acknowledgments.}
Jiaqi Yang is supported by National Natural Science Foundation of China under Grant: 12471225. 
%Authors used ChatGPT 5.6 for exploratory discussion and editorial feedback during the preparation of this manuscript. 
\section{Analysis of the linear operator}
\label{s:linear-similarity-operator}

We begin with the linear operator \(\mathcal L\) defined in
\eqref{LFDef}. The purpose of this section is to solve
\begin{equation} \label{LvFEqn}
  \mathcal L v = F.
\end{equation}

The dilation term requires some care at low frequency, so we first specify
the realization of the homogeneous Sobolev spaces used below.
For \(1<r<2\), a representative \(F\in\dot H^r(\R^2)\) is called normalized if its prescribed value \(F(0)\) and its homogeneous Littlewood--Paley blocks satisfy
\begin{equation}
\label{eq:normalized-representative-linear}
 F(x)-F(0)=\sum_{j\in\mathbb Z}\bigl(\dot\Delta_jF(x)-\dot\Delta_jF(0)\bigr).
\end{equation}
The series converges locally uniformly and obeys the homogeneous Morrey estimate
\begin{equation}
\label{eq:normalized-Morrey-linear}
 |F(x)-F(0)|\leq C_r|x|^{r-1}\|F\|_{\dot H^r}.
\end{equation}
Thus this convention excludes every nonconstant polynomial while retaining the prescribed additive constant.
All applications of \(\mathcal L^{-1}\) use this normalized realization.

The following proposition gives both the inverse formula and the estimate
needed in the nonlinear argument.

\begin{proposition}[Linear inverse in the cone-correction space]
\label{prop:linear-inverse-Sobolev}
Let \(\sigma,\mu>0\) and \(1<r<2\).
For every normalized representative \(F\in\dot H^r(\R^2)\), there exists a unique normalized solution
\[
 v\in\dot H^r(\R^2)\cap\dot H^{r+3}(\R^2)
\]
of \eqref{LvFEqn} in \(\mathcal S'(\R^2)\).
Moreover,
\begin{equation}
\label{eq:linear-inverse-Sobolev}
 \|v\|_{\dot H^r} + \|v\|_{\dot H^{r+3}} \leq C_{r,\sigma/\mu}\|F\|_{\dot H^r}.
\end{equation}
For almost every \(\rho>0\) and \(\omega\in\mathbb S^1\), its Fourier transform away from the origin is given by
\begin{equation}
\label{eq:linear-inverse-ray-formula}
 \widehat v(\rho\omega) = \rho^{-3}e^{-\frac{\sigma}{\mu}\rho^3} \int_0^\rho s^2e^{\frac{\sigma}{\mu}s^3}\widehat F(s\omega)\,ds.
\end{equation}
Equivalently, for every normalized \(F\in\dot H^r(\R^2)\),
\begin{equation}
\label{eq:linear-inverse-physical-formula}
 \mathcal L^{-1}F(x) = \int_0^1 e^{-\frac{\sigma}{\mu}(1-t^3)|D|^3}F(x/t)\,dt.
\end{equation}
The integral in \eqref{eq:linear-inverse-physical-formula} is understood in \(\mathcal S'(\R^2)\); the sublinear growth in \eqref{eq:normalized-Morrey-linear} makes it convergent at \(t=0\).
\end{proposition}

If two normalized representatives of the same homogeneous class differ by
a constant \(c\), then \eqref{eq:linear-inverse-physical-formula} gives
\[
 \mathcal L^{-1}(F+c)=\mathcal L^{-1}F+c.
\]
Thus \(\mathcal L^{-1}\) descends to a bounded inverse modulo constants.
For the nonlinear boundary expressions considered later, the datum is an
actual distribution, and the physical formula selects the corresponding
representative of the output.

\begin{proof}
We first assume that \(F\in\mathcal S(\R^2)\).
Taking the Fourier transform of \(\mathcal Lv=F\) gives
\begin{equation*}
 \xi\cdot\nabla_\xi\widehat v + \bigl(3+\frac{3\sigma}{\mu}|\xi|^3\bigr)\widehat v = \widehat F.
\end{equation*}
Along a ray \(\xi=\rho\omega\), this becomes
\[
 \rho\partial_\rho\widehat v(\rho\omega) + \bigl(3+\frac{3\sigma}{\mu}\rho^3\bigr)\widehat v(\rho\omega) = \widehat F(\rho\omega).
\]
Multiplication by the integrating factor \(\rho^3e^{\frac{\sigma}{\mu}\rho^3}\) gives
\[
 \partial_\rho \left( \rho^3e^{\frac{\sigma}{\mu}\rho^3}\widehat v(\rho\omega) \right) = \rho^2e^{\frac{\sigma}{\mu}\rho^3}\widehat F(\rho\omega).
\]
Choosing the particular solution with zero integration constant gives \eqref{eq:linear-inverse-ray-formula}.

Making the change of variables \(s=t\rho\) in \eqref{eq:linear-inverse-ray-formula}, we obtain
\[
 \widehat v(\xi) = \int_0^1 t^2e^{-\frac{\sigma}{\mu}(1-t^3)|\xi|^3} \widehat F(t\xi)\,dt.
\]
Since
\[
 \mathcal F\bigl[F(\,\cdot\,/t)\bigr](\xi) = t^2\widehat F(t\xi),
\]
this proves \eqref{eq:linear-inverse-physical-formula}.

We turn to the Sobolev estimate. Set \(a:=\sigma/\mu\) and, for a
scalar function \(f\) on \((0,\infty)\), define the integral operator
\[
 (\mathcal H_af)(\rho):=\rho^{-3}e^{-a\rho^3}
 \int_0^\rho s^2e^{as^3}f(s)\,ds.
\]
We claim that, for \(0<r<2\),
\begin{equation}
\label{eq:weighted-Hardy-linear}
 \int_0^\infty \bigl(\rho^{2r+1}+\rho^{2r+7}\bigr)
 |(\mathcal H_af)(\rho)|^2\,d\rho
 \leq C_{r,a}\int_0^\infty \rho^{2r+1}|f(\rho)|^2\,d\rho.
\end{equation}
It suffices to estimate the first weight on \(0<\rho<1\) and the second
on \(\rho\geq1\). On the first region, the weighted Hardy inequality
\[ 
\int_0^1 \rho^{-\lambda} \left| \int_0^\rho h(s)\,ds \right|^2d\rho \leq \frac{4}{(\lambda-1)^2} \int_0^1 s^{2-\lambda}|h(s)|^2\,ds,\quad \lambda>1,
\]
gives
\[
 \int_0^1\rho^{2r+1}|(\mathcal H_af)(\rho)|^2\,d\rho
 \lesssim_{r,a}\int_0^1s^{2r+1}|f(s)|^2\,ds.
\]

For \(\rho\geq1\), set \(g(s):=s^{r+1/2}f(s)\). 
Then we get
\[
 \rho^{r+7/2}(\mathcal H_af)(\rho)
 =\int_0^\rho K_r(\rho,s)g(s)\,ds,\qquad
 K_r(\rho,s):=\mathbf 1_{\{s<\rho\}}\rho^{r+1/2}s^{3/2-r}
 e^{-a(\rho^3-s^3)}.
\]
For the part \(0<s<1\), the Cauchy--Schwarz inequality gives
\begin{align*}
\int_1^\infty
 \left|
  \int_0^1K_r(\rho,s)g(s)\,ds
 \right|^2d\rho\leq
 \left(
  \int_1^\infty\int_0^1
  |K_r(\rho,s)|^2\,ds\,d\rho
 \right)
 \|g\|_{L^2(0,1)}^2.
\end{align*}
Moreover,
\[
 \int_1^\infty\int_0^1
 |K_r(\rho,s)|^2\,ds\,d\rho
 \lesssim_a
 \left(
  \int_1^\infty
  \rho^{2r+1}e^{-2a\rho^3}\,d\rho
 \right)
 \left(
  \int_0^1s^{3-2r}\,ds
 \right)
 <\infty,
\]
where \(r<2\) is used.

It remains to consider \(1\leq s<\rho\).  
If \(s\leq\rho/2\), the exponential gives rapid decay, while if
\(\rho/2<s<\rho\),
\[
 K_r(\rho,s) \lesssim_{r,a} \rho^2e^{-ca\rho^2(\rho-s)}.
\]
Consequently,
\[
 \sup_{\rho\geq1} \int_1^\rho K_r(\rho,s)\,ds
 + \sup_{s\geq1} \int_s^\infty K_r(\rho,s)\,d\rho \lesssim_{r,a}1.
\]
Thus the Cauchy--Schwarz inequality gives
\[
 \left| \int_1^\rho K_r(\rho,s)g(s)\,ds \right|^2
 \lesssim_{r,a} \int_1^\rho K_r(\rho,s)|g(s)|^2\,ds.
\]
Integrating in \(\rho\) and using Fubini's theorem, we obtain
\[
 \int_1^\infty \left|
  \int_1^\rho K_r(\rho,s)g(s)\,ds
 \right|^2d\rho \lesssim_{r,a} \int_1^\infty|g(s)|^2\,ds.
\]
Combining the two regions yields
\[
 \int_1^\infty \rho^{2r+7}|(\mathcal H_af)(\rho)|^2\,d\rho \lesssim_{r,a} \int_0^\infty s^{2r+1}|f(s)|^2\,ds,
\]
which proves the high-frequency estimate and hence
\eqref{eq:weighted-Hardy-linear}.

Applying this estimate to \(f_\omega(\rho):=\widehat F(\rho\omega)\)
and integrating in \(\omega\in\mathbb S^1\) yields
\[
 \|v\|_{\dot H^r}^2+\|v\|_{\dot H^{r+3}}^2
 \leq C_{r,\sigma/\mu}\|F\|_{\dot H^r}^2.
\]
This proves \eqref{eq:linear-inverse-Sobolev} for smooth \(F\).

The same formula yields the distributional identity. Indeed,
\[
 \left(\xi\cdot\nabla_\xi+3+3a|\xi|^3\right)
 \left[t^2e^{-a(1-t^3)|\xi|^3}\widehat F(t\xi)\right]
 =\partial_t\left[t^3e^{-a(1-t^3)|\xi|^3}\widehat F(t\xi)\right].
\]
Integration over \(0<t<1\) gives
\(\mathcal L\mathcal L^{-1}F=F\) in \(\mathcal S'\).

For a general normalized datum \(F\in\dot H^r\), choose by density
\(\varphi_N\in\mathcal S\) such that
\[
 F_N:=F(0)+\varphi_N-\varphi_N(0)\longrightarrow F
 \quad\text{in }\dot H^r.
\]
The Morrey estimate \eqref{eq:normalized-Morrey-linear} also gives local
uniform convergence. Since \eqref{eq:linear-inverse-physical-formula} maps
constants to constants, the preceding estimates apply to \(F_N\). Estimate
\eqref{eq:linear-inverse-Sobolev} makes
\(\mathcal L^{-1}F_N\) Cauchy in
\(\dot H^r\cap\dot H^{r+3}\). The kernels of
\(e^{-b|D|^3}\), \(0<b\leq a\), are
\(K_b(x)=b^{-2/3}K_1(b^{-1/3}x)\).  The decay argument used below for
\(K_1\) gives, for \(0\leq\gamma<1\),
\[
 \sup_{0<b\leq a}\left(
 \|K_b\|_{L^1}+\int_{\R^2}|x|^\gamma|K_b(x)|\,dx\right)<\infty.
\]
Together with the sublinear growth in \eqref{eq:normalized-Morrey-linear},
this permits passage to the limit in
\eqref{eq:linear-inverse-physical-formula} and identifies the Sobolev limit
with the normalized distribution defined by that formula.
At \(t=0\), the Fourier boundary term is the transform of
\(tF(\cdot/t)\), and
\[
 |tF(x/t)|\leq t|F(0)|+C_rt^{2-r}|x|^{r-1}\|F\|_{\dot H^r}\longrightarrow0.
\]
Thus \(\mathcal Lv=F\) in \(\mathcal S'\) for every normalized datum.

It remains to prove uniqueness. If \(\mathcal Lw=0\) and
\(w\in\dot H^r\), then away from the origin
\[
 \widehat w(\rho\omega)=\rho^{-3}e^{-a\rho^3}A(\omega),
 \qquad \rho>0.
\]
Since \(w\in\dot H^r\), its Fourier transform belongs to
\(L^2_{\mathrm{loc}}(\R^2\setminus\{0\})\).  In particular,
\[
 \int_1^2\int_{\mathbb S^1}
 |\widehat w(\rho\omega)|^2\,d\omega\,\rho\,d\rho
 <\infty.
\]
Hence, by Fubini's theorem, for almost every
\(\rho_0\in(1,2)\) one has
\(\widehat w(\rho_0\,\cdot)\in L^2(\mathbb S^1)\).
Evaluating the raywise formula at such a \(\rho_0\) gives
\[
 A(\omega)
 =
 \rho_0^3e^{a\rho_0^3}\widehat w(\rho_0\omega),
\]
and therefore \(A\in L^2(\mathbb S^1)\).
Moreover,
\begin{align*} 
\int_{0<|\xi|<1} 
|\xi|^{2r}|\widehat w(\xi)|^2\,d\xi 
&= 
\|A\|_{L^2(\mathbb S^1)}^2 
\int_0^1 
\rho^{2r-5}e^{-2a\rho^3}\,d\rho.
\end{align*}
The radial integral diverges because \(r<2\). 
Hence \(A=0\), so \(\widehat w\) vanishes on \(\R^2\setminus\{\mathbf{0}\}\) and is therefore supported at \(\xi= \mathbf{0}\).  
A tempered distribution supported at one point is a finite linear combination of derivatives of \(\delta_0\); consequently, \(w\) is a polynomial.
The normalization reduces \(w\) to a constant, and \(\mathcal Lc=c\) forces \(c=0\).
\end{proof}

We apply the linear analysis to the unbounded cone. 
To fix the polynomial ambiguity in the notation
\(e^{-\frac{\sigma}{\mu}|D|^3}|x|\), let \(c_2\neq0\) be the dimensional
constant for which \(\widehat{|x|}(\xi)=c_2|\xi|^{-3}\) away from
\(\xi= \mathbf{0}\), and define
\begin{equation}
\label{eq:regularized-cone-normalized-definition}
 u_*:=|x|+\mathcal F^{-1}\left[c_2\left(e^{-\frac{\sigma}{\mu}|\xi|^3}-1\right)|\xi|^{-3}\right].
\end{equation}
The Fourier multiplier inside the inverse transform belongs to \(L^1(\R^2)\).
Hence \eqref{eq:regularized-cone-normalized-definition} defines an actual tempered function and fixes the polynomial ambiguity.
We use \(e^{-\frac{\sigma}{\mu}|D|^3}|x|\) as shorthand for this normalized representative. 
Its scaling gives the distinguished homogeneous solution of the linear profile equation.

\begin{lemma}
\label{lem:regularized-cone-kernel}
Let $u_*=e^{-\frac{\sigma}{\mu}|D|^3}|x|$, then \(\mathcal Lu_*=0\) in \(\mathcal S'(\R^2)\).
\end{lemma}

\begin{proof}
For \(t>0\), define the normalized regularization at time \(t\) by
\[
 u_{*,t}:=|x|+\mathcal F^{-1}\left[c_2\left(e^{-\frac{\sigma}{\mu}t|\xi|^3}-1\right)|\xi|^{-3}\right].
\]
The same scaling calculation as in \eqref{eq:regularized-cone-normalized-definition} gives the actual distributional identity
\begin{equation*}
 u_{*,t}=t^{1/3}u_*(t^{-1/3}x)
\end{equation*}
in \(\mathcal S'(\R^2)\).
Moreover, the Fourier correction defining \(u_{*,t}\) is differentiable in \(\mathcal S'\) for \(t>0\), with
\[
 \partial_t u_{*,t}=-\frac{\sigma}{\mu}|D|^3u_{*,t}.
\]
Differentiating the scaling identity at \(t=1\) therefore yields
\[
 -\frac{\sigma}{\mu}|D|^3u_* = \frac13 \bigl( u_*-x\cdot\nabla u_* \bigr),
\]
which is exactly \(\mathcal Lu_*=0\).
\end{proof}

\section{The conical profile and the favorable quadratic cancellation}
\label{s:cone-capillary-expansion-revised}

Throughout the remainder of the paper, \(3/2<r<2\). We work in the
correction space
\[
 X_r:=\dot H^r(\R^2)\cap\dot H^{r+3}(\R^2),
\]
with the norm
\[
 \|v\|_{X_r} := \|v\|_{\dot H^r}+\|v\|_{\dot H^{r+3}}.
\]
As usual for homogeneous spaces, elements that differ by a constant are
identified. The normalization from
Section~\ref{s:linear-similarity-operator} excludes all other polynomial
components. Whenever an actual function is required, we use the
representative selected by the construction: \(u_*\) is fixed by
\eqref{eq:regularized-cone-normalized-definition}, and
\(\mathcal L^{-1}F\) by \eqref{eq:linear-inverse-physical-formula}. In
particular, the correction obtained from the fixed-point argument is an
actual tempered distribution, not only a class modulo constants.
For a possibly unbounded asymptotically conical graph \(h\), we define
\begin{equation*}
 \mathfrak N_r(h) := \|\nabla h\|_{L^\infty} +\|D^2h\|_{L^\infty} +\|h\|_{\dot H^{5/2}} +\sum_{k=1}^3\|h\|_{\dot H^{r+k}}.
\end{equation*}
Whenever \(G(h)\) or a pointwise norm of \(h\) is used, \(h\) is understood as a fixed Lipschitz representative, so no polynomial ambiguity is present.
Notice that \(\mathfrak N_r(h)\) requires neither \(h\in L^\infty\) nor \(h\in\dot H^r\).
The \(\dot H^{5/2}\)-norm provides the low-frequency energy control for
\(\Delta h\), whereas \(\dot H^{r+1}\) controls the coefficients of the
regularized flattening.

\subsection{The regularized cone and the correction space}

We first record the regularity of the distinguished conical profile.

\begin{lemma}[Bounds for the regularized cone]
\label{lem:regularized-cone-revised}
For every \(s>2\), one has \(u_*\in\dot H^s(\R^2)\).
Moreover, \(\nabla u_*,D^2u_*\in L^\infty(\R^2)\).
In particular, for \(3/2<r<2\),
\[
 \mathfrak N_r(u_*)<\infty.
\]
\end{lemma}

\begin{proof}
By the normalized Fourier representation,
\[
 \widehat u_*(\xi)=c_2e^{-\frac{\sigma}{\mu}|\xi|^3}|\xi|^{-3}
 \quad\text{away from }\xi= \mathbf{0}.
\]
Hence, for \(s>2\),
\[
 \|u_*\|_{\dot H^s}^2
 \simeq\int_0^\infty\rho^{2s-5}e^{-\frac{2\sigma}{\mu}\rho^3}\,d\rho<\infty;
\]
and the integral is finite precisely when \(s>2\).

Let
\(K_{\sigma/\mu}=\mathcal F^{-1}(e^{-\frac{\sigma}{\mu}|\xi|^3})\).
The symbol \(e^{-\frac{\sigma}{\mu}|\xi|^3}\) belongs to \(L^1\), hence
\(K_{\sigma/\mu}\in L^\infty\).
Its pure third weak derivatives also belong to \(L^1\); three integrations by parts therefore give
\[
 |K_{\sigma/\mu}(x)|\lesssim_{\sigma/\mu}(1+|x|)^{-3},
 \qquad K_{\sigma/\mu}\in L^1\cap L^\infty.
\]
Here no singular measure occurs at \(\xi=\mathbf{0}\), since the symbol
\(e^{-\frac{\sigma}{\mu}|\xi|^3}\) is \(C^2\) with locally Lipschitz
second derivatives.

Differentiating the normalized definition of \(u_*\), justified by
first mollifying \(|x|\), gives
\[
 D^ku_*=K_{\sigma/\mu}*D^k|x|,
 \qquad k=1,2.
\]
Since \(\nabla|x|\in L^\infty\), Young's inequality yields
\[
 \|\nabla u_*\|_{L^\infty}
 \leq
 \|K_{\sigma/\mu}\|_{L^1}\|\nabla|x|\|_{L^\infty}
 <\infty.
\]
Moreover,
\[
 D^2|x|
 =
 \frac1{|x|}
 \left(
  I-\frac{x\otimes x}{|x|^2}
 \right)
 \in
 L^1(B_1)\cap L^\infty(\R^2\setminus B_1).
\]
Splitting the convolution into \(|y|\leq1\) and \(|y|>1\), we obtain
\begin{align*}
 \|D^2u_*\|_{L^\infty}
 \leq
 \|K_{\sigma/\mu}\|_{L^\infty}
 \|D^2|x|\|_{L^1(B_1)}+
 \|K_{\sigma/\mu}\|_{L^1}
 \|D^2|x|\|_{L^\infty(\R^2\setminus B_1)}
 <\infty.
\end{align*}

Since \(5/2,r+1,r+2,r+3>2\), the Sobolev estimate above gives
\[
 u_*\in \dot H^{5/2} \cap\dot H^{r+1} \cap\dot H^{r+2} \cap\dot H^{r+3}.
\]
Together with the preceding \(L^\infty\)-bounds, this proves
\(\mathfrak N_r(u_*)<\infty\).
\end{proof}

The next lemma shows that the correction norm controls every quantity
appearing in \(\mathfrak N_r\).

\begin{lemma}[Compatibility of the correction space]
\label{lem:Xr-cone-embedding-revised}
Let \(3/2<r<2\).
For every normalized \(v\in X_r\),
\begin{equation*}
 \mathfrak N_r(v) \leq C_r\|v\|_{X_r}.
\end{equation*}
Consequently,
\begin{equation*}
 \mathfrak N_r(\varepsilon u_*+v) \leq C_{r,\sigma/\mu}\bigl(|\varepsilon|+\|v\|_{X_r}\bigr).
\end{equation*}
\end{lemma}

\begin{proof}
Homogeneous Sobolev interpolation gives
\[
 \|v\|_{\dot H^{5/2}}+ \|v\|_{\dot H^{r+1}}+ \|v\|_{\dot H^{r+2}} \lesssim_r\|v\|_{X_r}.
\]
For \(k=1,2\), Bernstein's inequality and a low--high split yield
\begin{align*}
 \|D^kv\|_{L^\infty} \lesssim \sum_{j\leq0}2^{(k+1-r)j}
   \bigl(2^{rj}\|\dot\Delta_jv\|_{L^2}\bigr)+\sum_{j>0}2^{(k-r-2)j}
   \bigl(2^{(r+3)j}\|\dot\Delta_jv\|_{L^2}\bigr)
 \lesssim_r\|v\|_{X_r}.
\end{align*}
Both geometric sums converge for \(k=1,2\) because \(r<2\) at low
frequency and \(r>0\) at high frequency.  
The same summation shows that the differentiated normalized Littlewood--Paley series converges uniformly.  
Combining the displayed estimates proves the first assertion.  
The second follows from the triangle inequality and Lemma~\ref{lem:regularized-cone-revised}.
\end{proof}

\subsection{The favorable quadratic cancellation}

We isolate the quadratic part of \(\mathscr F\). 
This is the only term for which the low-frequency structure of the cone requires a nonstandard estimate.

For Schwartz functions \(f,g\), the standard first shape-derivative
formula for the Dirichlet--Neumann operator at the flat interface is
\begin{equation}\label{eq:first-shape-derivative-flat}
 D_\eta G(0)[f]g = -|D|(f|D|g)-\operatorname{div}(f\nabla g).
\end{equation}
See, for example, \cite{Lannes05,Lannes11}. 
We define the bilinear forms
\begin{equation} \label{BB0Def}
 \mathcal B_0(f,g):= D_\eta G(0)[f]\Delta g, \qquad \mathcal B(f,g) := \frac12 \bigl( \mathcal B_0(f,g)+\mathcal B_0(g,f) \bigr).
\end{equation}
Then \(\mathcal B\) is symmetric and
\[
 \mathcal B(h,h)=\mathcal B_0(h,h).
\]

We shall use the following localized bilinear multiplier estimate.
For a bilinear symbol \(\sigma(\alpha,\beta)\), let
\(\mathcal M_\sigma\) be defined by
\[
 \widehat{\mathcal M_\sigma(F,G)}(\zeta)
 :=
 \int_{\R^2}
 \sigma(\zeta-\beta,\beta)
 \widehat F(\zeta-\beta)
 \widehat G(\beta)\,d\beta.
\]
If \(\sigma\) is dyadically localized and its rescaled
Coifman--Meyer seminorms are uniformly bounded, then its inverse
Fourier kernel has uniformly bounded \(L^1\)-norm.  Consequently,
\begin{equation*}
 \|\mathcal M_\sigma(F,G)\|_{L^2} \lesssim \|F\|_{L^\infty}\|G\|_{L^2}.
\end{equation*}
This is the localized form of the Coifman--Meyer multiplier estimate; see \cite{CM78}.  
The same statement holds for low--high symbols when the
low-frequency cutoff is summed before the estimate.

\begin{lemma}[Favorable quadratic cancellation]\label{lem:quadratic-cancellation-revised}
Let \(3/2<r<2\).
For \(f,g\in\mathcal S(\R^2)\), the Fourier symbol of
\(\mathcal{B}_0(f,g)\) in \eqref{BB0Def} is
\begin{equation*}
 b(\alpha,\beta) = \left( |\alpha+\beta||\beta| -(\alpha+\beta)\cdot\beta \right)|\beta|^2,
\end{equation*}
where \(\alpha\) and \(\beta\) are the frequencies of the first and second inputs, respectively.
It satisfies
\begin{equation*}
 |b(\alpha,\beta)| \leq \frac12|\alpha|^2|\beta|^2.
\end{equation*}
Moreover, in the high--high-to-low region
\[
 |\alpha|\sim|\beta|\sim N, \qquad |\alpha+\beta|\sim L\ll N,
\]
one has the symbolic estimate
\begin{equation*}
 |b(\alpha,\beta)| \lesssim LN^3.
\end{equation*}

The operator \(\mathcal B\), initially defined for Schwartz functions, has a canonical continuous extension to normalized representatives \(f,g\) satisfying
\[
	 f,g\in\dot H^{5/2}(\R^2)\cap\dot H^{r+2}(\R^2), \qquad D^2f,D^2g\in L^\infty(\R^2).
\]
The extension is obtained as the limit of the canonical normalized finite-frequency truncations described in the proof, belongs to \(\dot H^r(\R^2)\), and satisfies
\begin{align}\label{eq:B-estimate-revised}
 \|\mathcal B(f,g)\|_{\dot H^r}
 \lesssim_r{}&
 \|D^2f\|_{L^\infty}
 \|g\|_{\dot H^{r+2}}
 +
 \|f\|_{\dot H^{r+2}}
 \|D^2g\|_{L^\infty}.
\end{align}
\end{lemma}

\begin{proof}
For Schwartz inputs, \eqref{eq:first-shape-derivative-flat} gives, with
\(\zeta=\alpha+\beta\),
\[
 \widehat{\mathcal B_0(f,g)}(\zeta)
 =\int b(\alpha,\beta)\widehat f(\alpha)\widehat g(\beta)\,d\beta,
 \qquad
 b(\alpha,\beta)
 =\left(|\alpha+\beta||\beta|-(\alpha+\beta)\cdot\beta\right)|\beta|^2.
\]
Set
\(A=|\alpha+\beta||\beta|-(\alpha+\beta)\cdot\beta\).  Then
\[
 |\alpha|^2
 =\bigl(|\alpha+\beta|-|\beta|\bigr)^2+2A,
 \qquad 0\leq A\leq\tfrac12|\alpha|^2.
\]
This proves \(|b(\alpha,\beta)|\leq\frac12|\alpha|^2|\beta|^2\).
The alternative bound \(A\leq2|\alpha+\beta||\beta|\) gives
\(|b|\lesssim LN^3\) in the high--high-to-low region.

We prove the bilinear estimate by writing
\(f_j=\dot\Delta_jf\), \(g_k=\dot\Delta_kg\), and separating the usual
low--high, high--low, and high--high interactions.  All cutoffs below are
chosen with a fixed positive separation integer \(C\).

Suppose first that \(j\leq k-C\).  On the support of this interaction,
\(|\alpha|\ll|\beta|\) and \(|\alpha+\beta|\simeq|\beta|\).  Since
\(b(0,\beta)=\nabla_\alpha b(0,\beta)=0\), Taylor's formula shows that
\[
 b(\alpha,\beta)
 =\sum_{|\gamma|=|\delta|=2}
  \alpha^\gamma\beta^\delta m_{\gamma\delta}(\alpha,\beta),
\]
where, after restricting to \(|\alpha|\leq2^{-C}|\beta|\), the dyadically
rescaled symbols \(m_{\gamma\delta}(2^k\cdot,2^k\cdot)\) have uniformly
bounded Coifman--Meyer seminorms.  For the transposed term, the identity
\[
 \frac{b(\beta,\alpha)}{|\beta|^2|\alpha|^2}
 =\frac{|\alpha+\beta|}{|\beta|}\frac{|\alpha|}{|\beta|}
  -\frac{\beta}{|\beta|}\cdot\frac{\alpha}{|\beta|}
  -\frac{|\alpha|^2}{|\beta|^2}
\]
exhibits an additional factor \(|\alpha|/|\beta|\).  Rather than
differentiating the nonsmooth factor \(|\alpha|\) at the origin, we use the
physical-space identity for \(\mathcal B_0(g_k,f_j)\) and Bernstein's
inequality to obtain
\[
 \|\mathcal B_0(g_k,f_j)\|_{L^2}
 \lesssim 2^{j-k}\|D^2f_j\|_{L^\infty}\|D^2g_k\|_{L^2}.
\]
Since \(\|D^2f_j\|_{L^\infty}\lesssim\|D^2f\|_{L^\infty}\), summation
over \(j\leq k-C\) is geometric.  Combining this with the
Coifman--Meyer estimate for the nontransposed term gives
\begin{equation}
\label{eq:B-low-high-detail}
 \left\|\sum_{j\leq k-C}\mathcal B(f_j,g_k)\right\|_{L^2}
 \lesssim \|D^2f\|_{L^\infty}\|D^2g_k\|_{L^2}.
\end{equation}
Here the low-pass kernels are uniformly integrable, so no untruncated singular integral acts on an \(L^\infty\) function.  
Multiplying \eqref{eq:B-low-high-detail} by \(2^{rk}\), taking the \(\ell^2_k\)-norm, and then interchanging \(f,g\) controls both paraproduct regions by the right-hand side of \eqref{eq:B-estimate-revised}.

Next suppose \(|j-k|\leq C\) and the output frequency is \(2^q\) with \(q\geq j-C\).  
All nonzero frequencies are then comparable.  
Dividing the localized symbol by \(|\alpha|^2|\beta|^2\) produces a uniformly
Coifman--Meyer family; the possible nonsmoothness of
\(|\alpha+\beta|\) is excluded by the output cutoff.  Finite overlap in
\(j,k,q\) therefore gives the same two terms as in
\eqref{eq:B-estimate-revised}.

It remains to treat \(|j-k|\leq C\) and \(q\leq j-C\). The identity
\[
 \mathcal B_0(f,g)
 =|D|\bigl(f|D|^3g\bigr)-\operatorname{div}(f\nabla\Delta g)
\]
places one derivative on the low output.  Bernstein's inequality on the
annular inputs gives, term by term and also for the transpose,
\begin{align}
\label{eq:B-high-high-low-detail}
 \|\dot\Delta_q\mathcal B(f_j,g_k)\|_{L^2}
 \lesssim 2^{q-j}\bigl(
 \|D^2f_j\|_{L^\infty}\|D^2g_k\|_{L^2}
 +\|D^2f_j\|_{L^2}\|D^2g_k\|_{L^\infty}\bigr).
\end{align}
For example, the first product on the right arises from
\(2^q\|f_j\|_{L^\infty}2^{3k}\|g_k\|_{L^2}\), which is bounded by the
displayed quantity because \(|j-k|\leq C\).  
For \(j\in\mathbb Z\), set
\[
 a_j:=2^{rj}\sum_{|k-j|\leq C}\|D^2g_k\|_{L^2},
 \qquad \text{so that } \quad\|a\|_{\ell^2}\lesssim_r\|g\|_{\dot H^{r+2}}.
\]
After multiplying \eqref{eq:B-high-high-low-detail} by \(2^{rq}\), the
first term is bounded by
\[
 \|D^2f\|_{L^\infty}
 \sum_{j\geq q+C}2^{-(r+1)(j-q)}a_j.
\]
The convolution kernel belongs to \(\ell^1\), and discrete Young's
inequality gives the required \(\ell^2_q\)-bound.  
The second term is symmetric.  
This proves \eqref{eq:B-estimate-revised} for Schwartz inputs.

It remains to pass from Schwartz functions to the rough inputs needed
later. Let \(f,g\) satisfy
\[
 f,g\in\dot H^{5/2}(\R^2)\cap\dot H^{r+2}(\R^2),
 \qquad
 D^2f,D^2g\in L^\infty(\R^2).
\]
We use the normalized truncations
\[
 f^{(N)}(x)=f(0)+\sum_{|j|\leq N}
 \bigl(\dot\Delta_jf(x)-\dot\Delta_jf(0)\bigr),
\]
and similarly for \(g\). 
These truncations converge in
\(
 \dot H^{5/2}(\R^2)\cap\dot H^{r+2}(\R^2)
\).
Moreover, Bernstein's inequality and the Cauchy--Schwarz inequality
give
\begin{align*}
 \|D^2(f-f^{(N)})\|_{L^\infty}
 \lesssim \left(
  \sum_{j<-N}2^j \right)^{1/2} \|f\|_{\dot H^{5/2}}+ \left( \sum_{j>N}2^{2(1-r)j}
 \right)^{1/2} \|f\|_{\dot H^{r+2}}
 \longrightarrow 0.
\end{align*}
The same conclusions hold for \(g^{(N)}\). For \(M,N\geq1\), bilinearity gives
\begin{align*}
\mathcal B(f^{(N)},g^{(N)})
 - \mathcal B(f^{(M)},g^{(M)})=
 \mathcal B(f^{(N)}-f^{(M)},g^{(N)})
 + \mathcal B(f^{(M)},g^{(N)}-g^{(M)}).
\end{align*}
Applying \eqref{eq:B-estimate-revised}, together with the preceding
convergences and the uniform bounds on the truncations, shows that \(
 \bigl\{ \mathcal B(f^{(N)},g^{(N)}) \bigr\}_{N\geq1}
\)
is Cauchy in \(\dot H^r\). 
We therefore define
\[
 \mathcal B(f,g) := \lim_{N\to\infty} \mathcal B(f^{(N)},g^{(N)})
\]
as a homogeneous \(\dot H^r\)-class. 
Estimate \eqref{eq:B-estimate-revised} passes to the limit and gives the asserted continuous bilinear extension.

The identification of this homogeneous class with the quadratic form generated by \(D_\eta G(0)\), and the resulting choice of its tempered-distribution representative, are supplied by
Lemma~\ref{lem:quadratic-compatibility}.
\end{proof}

\section{Finite-order cone-adapted Dirichlet--Neumann estimates}
\label{s:finite-order-DN}

We establish the Dirichlet--Neumann estimates used in the nonlinear construction. 
Since the graph is asymptotically conical, the high Sobolev norm alone does not control the low-frequency energy. 
We therefore retain two lower-order norms of the Dirichlet datum.
For \(3/2<r<2\), set
\begin{equation*}
 \mathcal A_r(h) := \|\nabla h\|_{L^\infty} +\|h\|_{\dot H^{r+1}}, \qquad \mathcal D_r(f) := \|f\|_{\dot H^{1/2}} +\|f\|_{\dot H^1} +\|f\|_{\dot H^{r+1}}.
\end{equation*}
Let \(\mathfrak A_r\) denote the space of Lipschitz functions \(h\) with \(\mathcal A_r(h)<\infty\), modulo additive constants and equipped with the norm \(\mathcal A_r\). 
Let \(\mathfrak D_r\) be the completion of \(\mathcal S(\R^2)\) modulo constants under the norm \(\mathcal D_r\).
Equivalently, it is the compatible intersection of the three homogeneous Sobolev data classes appearing in \(\mathcal D_r\).

The proof is carried out in mixed low--high norms on the lower half-space.
If \((\dot\Delta_j)_{j\in\mathbb Z}\) is a homogeneous Littlewood--Paley decomposition, we define
\begin{equation*}
\begin{aligned}
 \|V\|_{\mathcal X_r} :={}& \|\nabla_{x,z}V\|_{L^\infty_zL^2_x} +\|\nabla_{x,z}V\|_{L^2_{x,z}}\\
 &+ \left( \sum_{j>0}2^{2rj} \|\dot\Delta_j\nabla_{x,z}V\|_{L^\infty_zL^2_x}^2 \right)^{1/2}+ \left( \sum_{j>0}2^{(2r+1)j} \|\dot\Delta_j\nabla_{x,z}V\|_{L^2_{x,z}}^2 \right)^{1/2},
\end{aligned}
\end{equation*}
where all the \(z\)-norms are taken on \(J=(-\infty,0)\).
For a vector field \(F\) on the lower half-space, let
\begin{equation*}
 \|F\|_{\mathcal Y_r} := \|F\|_{L^\infty_zL^2_x} + \|F\|_{L^2_{x,z}}+ \left( \sum_{j>0}2^{2rj} \|\dot\Delta_jF\|_{L^\infty_zL^2_x}^2 \right)^{1/2}+ \left( \sum_{j>0}2^{(2r+1)j} \|\dot\Delta_jF\|_{L^2_{x,z}}^2 \right)^{1/2}.
\end{equation*}
Thus \(\mathcal Y_r\) has the same low--high \(z\)-integrability pattern as the gradient norm in \(\mathcal X_r\).
Only the positive horizontal frequencies carry a high Sobolev weight; this is essential for conical graphs.

\begin{equation*}
 \|A\|_{\mathcal E_r} :=\|A\|_{L^\infty_{x,z}}+ \left( \sum_{j\in\mathbb Z}2^{2rj} \|\dot\Delta_jA\|_{L^\infty_zL^2_x}^2 \right)^{1/2}+ \left( \sum_{j\in\mathbb Z}2^{(2r+1)j} \|\dot\Delta_jA\|_{L^2_{x,z}}^2 \right)^{1/2}.
\end{equation*}
For vector- and matrix-valued functions this norm is understood componentwise.
The homogeneous dyadic terms are taken modulo polynomials, while the \(L^\infty\)-representative fixes the coefficients pointwise.
In particular, constants belong to \(\mathcal E_r\).

We begin with the algebra and composition properties of the coefficient space.

\begin{lemma}[Coefficient algebra and composition]
\label{lem:coefficient-algebra}
Let \(r>1\).
Then \(\mathcal E_r\) is a Banach algebra and
\begin{equation*}
 \|AB\|_{\mathcal E_r} \leq C_r \left( \|A\|_{L^\infty}\|B\|_{\mathcal E_r} + \|B\|_{L^\infty}\|A\|_{\mathcal E_r} \right).
\end{equation*}

Let \(\Phi\) be a \(C^\infty\) function on a neighborhood of the closed ball \(\{p:|p|\leq M\}\) in a finite-dimensional vector space, and define the Nemytskii map
\[
 \mathcal N_\Phi(A):=\Phi\circ A.
\]
Then \(\mathcal N_\Phi\) is \(C^4\) from 
\[
 \left\{ A\in\mathcal E_r: \|A\|_{L^\infty}<M \right\}
\]
to \(\mathcal E_r\).
More precisely, on every set on which
\[
 \|A\|_{\mathcal E_r}\leq M_1, \qquad \|A\|_{L^\infty}\leq M_0<M,
\]
and for \(1\leq m\leq4\), one has
\begin{equation*}
 \left\| D^m\mathcal N_\Phi(A)[B_1,\ldots,B_m] \right\|_{\mathcal E_r} \leq C_{\Phi,M_0,M_1,r,m} \prod_{\nu=1}^m \|B_\nu\|_{\mathcal E_r}.
\end{equation*}
\end{lemma}

\begin{proof}
For \(s>0\) and \(q\in\{2,\infty\}\), we set
\[
 \|U\|_{\widetilde L^q_z\dot H^s_x} := \left( \sum_{j\in\mathbb Z} 2^{2sj} \|\dot\Delta_jU\|_{L^q_zL^2_x}^2 \right)^{1/2}.
\]
Thus, the \(\mathcal E_r\)-norm satisfies
\[
 \|A\|_{\mathcal E_r} = \|A\|_{L^\infty_{x,z}} + \|A\|_{\widetilde L^\infty_z\dot H^r_x} + \|A\|_{\widetilde L^2_z\dot H^{r+1/2}_x}.
\]
The horizontal Bony decomposition gives, for every \(s>0\),
\[
 \|AB\|_{\widetilde L^q_z\dot H^s_x}
 \lesssim_s \|A\|_{L^\infty}\|B\|_{\widetilde L^q_z\dot H^s_x}
 +\|B\|_{L^\infty}\|A\|_{\widetilde L^q_z\dot H^s_x}.
\]
For the resonant term, the high--high-to-low summation has kernel
\(2^{-s(j-m)}\mathbf1_{\{j\geq m-C\}}\in\ell^1\); the paraproducts are immediate.  
Taking \((s,q)=(r,\infty)\) and
\((r+\frac12,2)\), together with the pointwise product estimate, proves the algebra bound.  
Completeness follows componentwise in the three norms, with the \(L^\infty\)-representative fixing the homogeneous ambiguity.

The same Bony argument gives the mixed homogeneous Moser estimate
\[
 \|\Psi(A)-\Psi(0)\|_{\widetilde L^q_z\dot H^s_x}
 \leq C_{\Psi,\|A\|_{L^\infty},s}
 \|A\|_{\widetilde L^q_z\dot H^s_x}.
\]
Applying it to \(D^m\Phi(A)\), \(1\leq m\leq4\), and using the algebra estimate repeatedly in the pointwise identity
\[
 D^m\mathcal N_\Phi(A)[B_1,\ldots,B_m]
 =D^m\Phi(A)[B_1,\ldots,B_m],
\]
we obtain the stated bounds.  
The mean-value formula makes these derivatives locally Lipschitz in \(A\), and Taylor's formula with integral remainder proves the required Fréchet differentiability through order four.
\end{proof}

The flattening map will be built from a scale-dependent regularization of the graph. 
The next lemma records the corresponding coefficient and trace bounds.

\begin{lemma}[Regularized extension bounds]
\label{lem:regularized-extension-two-level}
Fix an even function \(\chi\) such that
\[
 \chi(x)\in C^\infty_c(\R^2), \qquad \int_{\R^2}\chi(x) dx =1,
\]
and define its rescaling by
\[
 \chi_\tau(x):=\tau^{-2}\chi(x/\tau), \qquad \tau>0.
\]
For \(h\in\mathfrak A_r\) and \(z<0\), we define
\[
 H_h(x,z):=\chi_{-z}*h(x), \qquad E_h:=\nabla_{x,z}H_h.
\]
We also set \(H_h(x,0)=h(x)\) in the trace sense.
Then
\begin{equation}
\label{eq:Eh-Linfty}
 \|E_h\|_{L^\infty_{x,z}} \leq C_\chi\|\nabla h\|_{L^\infty}.
\end{equation}
Moreover, for every \(N\geq0\),
\begin{equation*}
 \|\dot\Delta_jE_h(\cdot,z)\|_{L^2} \leq C_{N,\chi} 2^j(1+2^j|z|)^{-N} \|\dot\Delta_jh\|_{L^2}.
\end{equation*}
Consequently,
\begin{align}
\label{eq:Eh-high}
 \left( \sum_{j\in\mathbb Z} 2^{2rj} \|\dot\Delta_jE_h\|_{L^\infty_zL^2_x}^2 \right)^{1/2}+ \left( \sum_{j\in\mathbb Z} 2^{(2r+1)j} \|\dot\Delta_jE_h\|_{L^2_{x,z}}^2 \right)^{1/2} \leq C_{r,\chi}\|h\|_{\dot H^{r+1}}.
\end{align}
In particular,
\begin{equation*}
 \|E_h\|_{\mathcal E_r} \leq C_{r,\chi}\mathcal A_r(h).
\end{equation*}
Finally, as \(z\uparrow0\),
\begin{equation}
\label{eq:regularized-extension-boundary-traces}
\begin{aligned}
 H_h(\cdot,z)&\longrightarrow h \qquad \text{locally uniformly and in }\dot H^{r+1},\\
 \nabla_xH_h(\cdot,z)&\longrightarrow\nabla h \quad\text{in }\dot H^r,
 \qquad \partial_zH_h(\cdot,z)\longrightarrow0 \quad\text{in }\dot H^r.
\end{aligned}
\end{equation}
\end{lemma}

\begin{proof}
Set \(\tau:=-z>0\).
The tangential estimate follows from \(\nabla_xH_h=\chi_\tau*\nabla h\). 
Since \(\int\partial_\tau\chi_\tau =0\),
\[
 \partial_zH_h(x,z)
 =-\int\partial_\tau\chi_\tau(y)\bigl(h(x-y)-h(x)\bigr)\,dy,
\]
and scaling gives
\(\int |y||\partial_\tau\chi_\tau(y)|\,dy\lesssim_\chi1\).
This proves \eqref{eq:Eh-Linfty}.

On the Fourier side,
\[
 \widehat{\nabla_xH_h}=i\xi\widehat\chi(\tau\xi)\widehat h, \qquad
 \widehat{\partial_zH_h}=-\xi\cdot\nabla\widehat\chi(\tau\xi)\widehat h.
\]
The Schwartz decay of \(\widehat\chi\) therefore yields
\[
 \|\dot\Delta_jE_h(\cdot,z)\|_{L^2}
 \lesssim_{N,\chi}2^j(1+2^j|z|)^{-N}
 \|\dot\Delta_jh\|_{L^2}.
\]
Taking the \(L^\infty_z\)-norm, or the \(L^2_z\)-norm and using
\(\|(1+2^j|z|)^{-N}\|_{L^2_z}\lesssim_N2^{-j/2}\), gives
\eqref{eq:Eh-high} after summation.  Together with
\eqref{eq:Eh-Linfty}, this proves
\(\|E_h\|_{\mathcal E_r}\lesssim\mathcal A_r(h)\).

Finally,
\[
 |H_h(x,-\tau)-h(x)|\lesssim_\chi
 \tau\|\nabla h\|_{L^\infty},
\]
while \(\widehat\chi(\tau\xi)\to1\).  Dominated convergence gives the
first two limits in \eqref{eq:regularized-extension-boundary-traces}.
Evenness of \(\chi\) implies \(\nabla\widehat\chi(0)=0\); applying dominated
convergence to \(-\xi\cdot\nabla\widehat\chi(\tau\xi)\widehat h\) gives the
normal-trace limit.
\end{proof}

After factoring the flat Laplacian into first-order operators, the basic elliptic estimate reduces to the following Duhamel bounds.

\begin{lemma}[Flat mixed first-order estimate]
\label{lem:flat-mixed-elliptic}
Let \(f\) satisfy \(\mathcal D_r(f)<\infty\), and let \(F_0,F_1\in\mathcal Y_r\).
Consider
\begin{align*}
 (\partial_z-|D|)V&=W+F_0, &V|_{z=0}&=f,\\
 (\partial_z+|D|)W&=|D|F_1, &\lim_{z\to-\infty}W(z)&=0.
\end{align*}
Then
\begin{equation*}
 \|V\|_{\mathcal X_r} + \|W\|_{\mathcal Y_r} + \|W|_{z=0}\|_{\dot H^r} \leq C_r \left( \mathcal D_r(f) + \|F_0\|_{\mathcal Y_r} + \|F_1\|_{\mathcal Y_r} \right).
\end{equation*}
\end{lemma}

\begin{proof}
For the Poisson extension \(P[f](z)=e^{z|D|}f\), Plancherel gives
\[
 \|\dot\Delta_j\nabla P[f]\|_{L^\infty_zL^2_x}
 \lesssim2^j\|\dot\Delta_jf\|_{L^2},\qquad
 \|\dot\Delta_j\nabla P[f]\|_{L^2_{x,z}}
 \lesssim2^{j/2}\|\dot\Delta_jf\|_{L^2}.
\]
Summing the unweighted and positive-frequency components yields
\[
 \|P[f]\|_{\mathcal X_r}\lesssim_r\mathcal D_r(f).
\]
Indeed, the unweighted square sums are controlled by
\(\|f\|_{\dot H^1}\) and \(\|f\|_{\dot H^{1/2}}\), while for \(j>0\)
both weighted components reduce to
\(2^{(r+1)j}\|\dot\Delta_jf\|_{L^2}\).

The second equation is represented by
\[
 W(z)=\int_{-\infty}^z e^{-(z-z')|D|}|D|F_1(z')\,dz'.
\]
At frequency \(2^j\), its \(z\)-kernel is bounded by
\(k_j(s)=C2^je^{-c2^js}\mathbf1_{\{s\geq0\}}\), with
\(\|k_j\|_{L^1}\lesssim1\) and \(\|k_j\|_{L^2}\lesssim2^{j/2}\).
Thus, for \(p=2,\infty\),
\[
 \|\dot\Delta_jW\|_{L^p_zL^2_x}\lesssim
 \|\dot\Delta_jF_1\|_{L^p_zL^2_x},\qquad
 \|\dot\Delta_jW\|_{L^\infty_zL^2_x}\lesssim
 2^{j/2}\|\dot\Delta_jF_1\|_{L^2_{x,z}}.
\]
For the unweighted \(L^\infty_zL^2_x\)-norm, almost orthogonality and a
split at frequency one give
\begin{align*}
 \|W\|_{L^\infty_zL^2_x}^2
 \lesssim{}&
 \sum_{j\leq0}2^j\|\dot\Delta_jF_1\|_{L^2_{x,z}}^2
 +\sum_{j>0}\|\dot\Delta_jF_1\|_{L^\infty_zL^2_x}^2
 \lesssim \|F_1\|_{\mathcal Y_r}^2.
\end{align*}
The first dyadic estimate with \(p=2\) controls the unweighted bulk
norm.  For \(j>0\), multiplying it by \(2^{rj}\) or
\(2^{(r+1/2)j}\) controls the two weighted components after taking the
corresponding \(\ell^2\)-norm.

The same \(L^2_z\)-to-boundary estimate gives
\[
 \|\dot\Delta_jW(0)\|_{L^2_x}\lesssim
 2^{j/2}\|\dot\Delta_jF_1\|_{L^2_{x,z}}.
\]
Consequently,
\begin{align*}
 \|W(0)\|_{\dot H^r}^2\lesssim \sum_{j\leq0}2^{(2r+1)j}\|\dot\Delta_jF_1\|_{L^2_{x,z}}^2
+\sum_{j>0}2^{(2r+1)j}\|\dot\Delta_jF_1\|_{L^2_{x,z}}^2\lesssim_r\|F_1\|_{\mathcal Y_r}^2,
\end{align*}
because \(2^{(2r+1)j}\leq1\) for \(j\leq0\).  We have proved
\[
 \|W\|_{\mathcal Y_r}+\|W(0)\|_{\dot H^r}
 \lesssim_r\|F_1\|_{\mathcal Y_r}.
\]

Set \(U=V-P[f]\) and \(S=W+F_0\).  Then
\[
 U(z)=\int_0^z e^{(z-z')|D|}S(z')\,dz',\qquad U(0)=0.
\]
The operator from \(S\) to \(|D|U\) has the same dyadic kernels as
above. Hence every component of the \(\mathcal X_r\)-norm of \(|D|U\)
is bounded by \(\|S\|_{\mathcal Y_r}\). The boundedness of the Riesz
transforms gives the corresponding estimate for \(\nabla_xU\), while
\(\partial_zU=|D|U+S\). Therefore
\[
 \|U\|_{\mathcal X_r}\lesssim_r\|S\|_{\mathcal Y_r}  \lesssim_r\|F_0\|_{\mathcal Y_r}+\|F_1\|_{\mathcal Y_r}.
\]
Combining this with the Poisson, \(W\), and trace bounds proves the lemma.
\end{proof}

To handle the variable coefficients, we also need a product estimate in
the mixed norms.

\begin{lemma}[Mixed tame product estimate]
\label{lem:mixed-tame-product}
Let \(r>1\).
Suppose that \(a\in L^\infty(\R^3_-)\) and
\[
 \|a\|_{\mathcal E_r^+} := \left( \sum_{j>0} 2^{2rj} \|\dot\Delta_ja\|_{L^\infty_zL^2_x}^2 \right)^{1/2} + \left( \sum_{j>0} 2^{(2r+1)j} \|\dot\Delta_ja\|_{L^2_{x,z}}^2 \right)^{1/2} <\infty.
\]
If \(B=\partial_\ell V\), where
\(\partial_\ell\in
\{\partial_{x_1},\partial_{x_2},\partial_z\}\), then
\begin{equation}
\label{eq:mixed-tame-product}
 \|aB\|_{\mathcal Y_r} \leq C_r \left( \|a\|_{L^\infty}\|V\|_{\mathcal X_r} + \|a\|_{\mathcal E_r^+} \|\nabla_{x,z}V\|_{L^\infty_{x,z}} \right).
\end{equation}
Moreover,
\begin{equation}
\label{eq:X-to-Linfty}
 \|\nabla_{x,z}V\|_{L^\infty_{x,z}} \leq C_r\|V\|_{\mathcal X_r}.
\end{equation}

If \(a\) is a finite product or a smooth composition of coefficients belonging to \(\mathcal E_r\), then \eqref{eq:mixed-tame-product} remains valid after estimating \(\|a\|_{L^\infty}+\|a\|_{\mathcal E_r^+}\) by Lemma~\ref{lem:coefficient-algebra}.
\end{lemma}

\begin{proof}
For \(s>0\) and \(q\in\{2,\infty\}\), define the positive-frequency mixed norm
\[
 \|F\|_{\widetilde L^q_z\dot H^s_{x,+}} := \left( \sum_{j>0} 2^{2sj} \|\dot\Delta_jF\|_{L^q_zL^2_x}^2 \right)^{1/2}.
\]
The horizontal Bony decomposition gives the hybrid estimate
\begin{equation*}
 \|aB\|_{\widetilde L^q_z\dot H^s_{x,+}} \lesssim_s \|a\|_{L^\infty_{x,z}} \left( \|B\|_{L^q_zL^2_x} + \|B\|_{\widetilde L^q_z\dot H^s_{x,+}} \right)+ \|B\|_{L^\infty_{x,z}} \|a\|_{\widetilde L^q_z\dot H^s_{x,+}}.
\end{equation*}
Indeed, for \(T_aB=\sum_jS_{j-C}a\,\dot\Delta_jB\), a positive output
frequency \(2^m\) requires \(j=m+O(1)\). Thus
\(T_aB\) is controlled by \(\|a\|_{L^\infty}\) times the
positive-frequency norm of \(B\); if \(j\leq0\), only finitely many
positive output blocks occur and are controlled by the unweighted
norm of \(B\). The term \(T_Ba\) is treated similarly, placing \(B\)
in \(L^\infty\). For the resonant term, the input frequency \(2^j\) and output
frequency \(2^m\) satisfy \(m\leq j+C\). Hence
\begin{align*}
2^{sm}\|\dot\Delta_mR(a,B)\|_{L^q_zL^2_x} 
\lesssim 
\sum_{j\geq m-C}2^{-s(j-m)} 
\Bigl(
\|a\|_{L^\infty}2^{sj}\|\dot\Delta_jB\|_{L^q_zL^2_x}+  
\|B\|_{L^\infty}  
2^{sj}\|\dot\Delta_ja\|_{L^q_zL^2_x} 
\Bigr).
\end{align*}
The kernel \(2^{-s(j-m)}\mathbf 1_{\{j\geq m-C\}}\) belongs to
\(\ell^1\) because \(s>0\), so discrete Young's inequality gives the
claimed hybrid estimate.

For the two unweighted components of the \(\mathcal Y_r\)-norm, Hölder's inequality gives
\begin{equation*}
 \|aB\|_{L^\infty_zL^2_x} + \|aB\|_{L^2_{x,z}} \leq \|a\|_{L^\infty} \left( \|B\|_{L^\infty_zL^2_x} + \|B\|_{L^2_{x,z}} \right).
\end{equation*}
Applying the hybrid bound with \((s,q)=(r,\infty)\) and
\((r+\frac12,2)\) proves \eqref{eq:mixed-tame-product}.

For \eqref{eq:X-to-Linfty}, Bernstein's inequality at fixed \(z<0\) gives
\begin{align*}
 \|\nabla V(z)\|_{L^\infty_x}
 \lesssim{}&
 \left(\sum_{j\leq0}2^{2j}\right)^{1/2}\|\nabla V(z)\|_{L^2_x}+\left(\sum_{j>0}2^{2(1-r)j}\right)^{1/2}
 \left(\sum_{j>0}2^{2rj}\|\dot\Delta_j\nabla V(z)\|_{L^2_x}^2\right)^{1/2}.
\end{align*}
Taking the supremum in \(z\) proves \eqref{eq:X-to-Linfty}; the positive
sum is finite because \(r>1\).

The final assertion follows from
\(\|a\|_{\mathcal E_r^+}\leq\|a\|_{\mathcal E_r}\) and the product and
composition estimates of Lemma~\ref{lem:coefficient-algebra}.
\end{proof}

We can now prove the finite-order estimate for the Dirichlet-Neumann operator that will be used throughout the remainder of the paper.

\begin{proposition}[Finite-order cone-adapted DN estimate]
\label{prop:finite-order-DN}
Let \(3/2<r<2\).
There exists \(\delta_r>0\) such that the map
\[
 \eta\longmapsto G(\eta)
\]
is \(C^3\) from the open ball
\[
 \{\eta\in\mathfrak A_r:\mathcal A_r(\eta)<2\delta_r\}
\]
into \(\mathcal L(\mathfrak D_r,\dot H^r(\R^2))\).
More precisely, whenever \(\mathcal A_r(\eta)\leq\delta_r\) and
\(k=0,1,2,3\),
\begin{equation}
\label{eq:finite-order-DN}
 \left\| D_\eta^kG(\eta)[h_1,\ldots,h_k]f \right\|_{\dot H^r} \leq C_{r,k}\, \mathcal D_r(f) \prod_{\nu=1}^k\mathcal A_r(h_\nu).
\end{equation}
For \(k=0\), the product on the right-hand side is understood to be one, and the constants are uniform under the stated smallness condition.
\end{proposition}

\begin{proof}
We first work with smooth finite-frequency data; the estimates are uniform
and therefore extend to the completed spaces. Set
\[
 \rho_\eta(x,z)=z+H_\eta(x,z),
 \qquad v_\eta(x,z)=\phi_f(x,\rho_\eta(x,z)).
\]
For \(\delta_r\) small enough, Lemma~\ref{lem:regularized-extension-two-level} gives
\[
\tfrac12\leq\partial_z\rho_\eta\leq\tfrac32, \qquad |H_\eta(x,z)-\eta(x)|
 \lesssim |z|\|\nabla\eta\|_{L^\infty}.
\]
Together with the boundary traces, these estimates show that \((x,z)\mapsto(x,\rho_\eta(x,z))\) is a bi-Lipschitz map from the lower
half-space onto \(\Omega_\eta\).

If we write \(E_\eta:=\nabla_{x,z}H_\eta\), a direct change-of-variables gives
\begin{equation*}
 \operatorname{div}_{x,z} \left( \mathsf A_\eta\nabla_{x,z}v_\eta \right)=0, \qquad
 \mathsf A_\eta  :=
 \begin{pmatrix}
 (1+\partial_zH_\eta)I_2 & -\nabla H_\eta\\[2mm]
 -(\nabla H_\eta)^T & \displaystyle \frac{1+|\nabla H_\eta|^2} {1+\partial_zH_\eta}
 \end{pmatrix}.
\end{equation*}
The matrix \(\mathsf A_\eta\) is uniformly elliptic.
Indeed, if \(a=\nabla H_\eta\), \(b=\partial_zH_\eta\), and \((\xi,\zeta)\in\R^2\times\R\), then
\[
 \mathsf A_\eta \binom{\xi}{\zeta} \cdot \binom{\xi}{\zeta} = \frac1{1+b} \left( |(1+b)\xi-a\zeta|^2+\zeta^2 \right),
\]
and \(a,b\) are uniformly small.

Expanding the divergence-form equation gives the equivalent system
\begin{equation*}
 \Delta_{x,z}v_\eta = \partial_zQ_a(\eta,v_\eta) + |D|Q_b(\eta,v_\eta), \qquad v_\eta|_{z=0}=f,
\end{equation*}
where
\begin{align*}
 Q_a(\eta,V) &:= \nabla H_\eta\cdot\nabla_xV
 -\frac{|\nabla H_\eta|^2-\partial_zH_\eta}
 {1+\partial_zH_\eta}\,\partial_zV,\\
 Q_b(\eta,V) &:= |D|^{-1}\operatorname{div}_x
 \left(  \nabla H_\eta\,\partial_zV
  -\partial_zH_\eta\,\nabla_xV \right).
\end{align*}

Define
\[
 w_\eta := (\partial_z-|D|)v_\eta - Q_a(\eta,v_\eta).
\]
Using the decomposition of the Laplacian
\[
 \Delta_{x,z} = (\partial_z+|D|)(\partial_z-|D|),
\]
we obtain
\begin{equation*}
 (\partial_z+|D|)w_\eta = |D| \left( Q_b(\eta,v_\eta)-Q_a(\eta,v_\eta) \right).
\end{equation*}

At \(z=0\), put \(a=\nabla_xH_\eta=\nabla\eta\) and
\(b=\partial_zH_\eta\).  The chain rule gives
\[
 G(\eta)f
 =-a\cdot\nabla_xv_\eta+\frac{1+|a|^2}{1+b}\partial_zv_\eta
 =\bigl(\partial_zv_\eta-Q_a(\eta,v_\eta)\bigr)\big|_{z=0}.
\]
Since \(v_\eta(0)=f\),
\begin{equation}
\label{eq:DN-trace-first-order}
 G(\eta)f=|D|f+w_\eta(0).
\end{equation}

Writing \(R_j:=|D|^{-1}\partial_{x_j}\), we have
\[
 Q_b(\eta,V) = \sum_{j=1}^2 R_j
 \left(  \partial_{x_j}H_\eta\,\partial_zV
  -\partial_zH_\eta\,\partial_{x_j}V
 \right).
\]
Thus \(Q_a\), and the expressions inside the Riesz transforms in
\(Q_b\), are sums of products \(C(E_\eta)\partial_\ell V\), where
\(C\) is smooth and \(C(0)=0\). 
Lemma~\ref{lem:coefficient-algebra} and
Lemma~\ref{lem:regularized-extension-two-level} give
\[
 \|C(E_\eta)\|_{L^\infty} + \|C(E_\eta)\|_{\mathcal E_r^+} \lesssim_r \mathcal A_r(\eta).
\]
Lemma~\ref{lem:mixed-tame-product} and the \(L^2\)-boundedness of the horizontal Riesz transforms therefore yield
\begin{equation}\label{eq:Q-base-bound}
 \|Q_a(\eta,V)\|_{\mathcal Y_r}
 + \|Q_b(\eta,V)\|_{\mathcal Y_r}
 \lesssim_r \mathcal A_r(\eta)\|V\|_{\mathcal X_r}.
\end{equation}

For \(V\in\mathcal X_r\), define
\begin{equation*}
 \mathcal W_\eta[V](z) := \int_{-\infty}^z e^{-(z-z')|D|}|D| \left( Q_b(\eta,V)-Q_a(\eta,V) \right)(z')\,dz'
\end{equation*}
and
\begin{equation*}
 \mathcal K_\eta V(z) := \int_0^z e^{(z-z')|D|} \left( \mathcal W_\eta[V](z') + Q_a(\eta,V)(z') \right)\,dz'.
\end{equation*}
Applying Lemma~\ref{lem:flat-mixed-elliptic} with zero boundary datum,
\(F_0=Q_a(\eta,V)\), and \(F_1=Q_b(\eta,V)-Q_a(\eta,V)\), and then using
\eqref{eq:Q-base-bound}, we obtain
\begin{equation} \label{eq:K-W-base-bound}
\begin{aligned}
 \|\mathcal K_\eta V\|_{\mathcal X_r} + \|\mathcal W_\eta[V]\|_{\mathcal Y_r} + \|\mathcal W_\eta[V](0)\|_{\dot H^r}\leq C_r\mathcal A_r(\eta) \|V\|_{\mathcal X_r}.
\end{aligned}
\end{equation}

The flattened extension satisfies
\begin{equation*}
 v_\eta = P[f]+\mathcal K_\eta v_\eta, \qquad P[f](z):=e^{z|D|}f.
\end{equation*}
Choose \(\delta_r\) so that \(\|\mathcal K_\eta\| \leq\frac12\), then $(I-\mathcal K_\eta)$ is invertible and
\[
 v_\eta = (I-\mathcal K_\eta)^{-1}P[f].
\]
The Poisson estimate and \eqref{eq:K-W-base-bound} give
\begin{equation*}
 \|v_\eta\|_{\mathcal X_r} + \|w_\eta\|_{\mathcal Y_r} + \|w_\eta(0)\|_{\dot H^r} \leq C_r\mathcal D_r(f).
\end{equation*}
The trace formula \eqref{eq:DN-trace-first-order} proves
\eqref{eq:finite-order-DN} for \(k=0\).

The map \(h\mapsto E_h\) is linear from \(\mathfrak A_r\) to \(\mathcal E_r\).  
Since \(1+\partial_zH_\eta\geq\frac12\), the coefficient algebra, composition, and mixed product estimates give, for
\(1\leq m\leq4\),
\begin{equation}
\label{eq:Q-shape-operator-bound}
 \sum_{i = a, b}\left\| D_\eta^mQ_i(\eta,V) [h_1,\ldots,h_m] \right\|_{\mathcal Y_r}\leq C_{r,m} \left( \prod_{\nu=1}^m \mathcal A_r(h_\nu) \right) \|V\|_{\mathcal X_r}.
\end{equation}
Each differentiated term consists of \(E_{h_1},\ldots,E_{h_m}\), a smooth bounded function of \(E_\eta\), and one component of \(\nabla V\);
the tame estimates place at most one factor in a homogeneous norm.

Define \(\mathcal Z_\eta V:=\mathcal W_\eta[V](0)\).
Because the two Duhamel operators in Lemma~\ref{lem:flat-mixed-elliptic} are bounded and independent of \(\eta\), differentiating them under the integral sign and using
\eqref{eq:Q-shape-operator-bound} shows that \(\mathcal K_\eta\) and \(\mathcal Z_\eta\) are \(C^3\) in operator norm.  
More precisely, for
\(1\leq m\leq3\),
\begin{equation}
\label{eq:K-shape-bound}
\begin{aligned}
 \left\| D_\eta^m\mathcal K_\eta [h_1,\ldots,h_m]V \right\|_{\mathcal X_r}+ \left\| D_\eta^m\mathcal Z_\eta [h_1,\ldots,h_m]V \right\|_{\dot H^r}\leq C_{r,m} \left( \prod_{\nu=1}^m \mathcal A_r(h_\nu) \right) \|V\|_{\mathcal X_r}.
\end{aligned}
\end{equation}
For example, \(D_\eta^m\mathcal W_\eta[V]\) is obtained by replacing \(Q_b-Q_a\) in its defining integral by \(D_\eta^m(Q_b-Q_a)[h_1,\ldots,h_m]V\), and the formula for
\(D_\eta^m\mathcal K_\eta\) has the analogous two terms.  
The estimate with \(m=4\) following from \eqref{eq:Q-shape-operator-bound} controls the third-order Taylor remainder and proves continuity of the third derivatives.

Set
\[
 \mathcal R_\eta:=(I-\mathcal K_\eta)^{-1}.
\]
By \eqref{eq:K-W-base-bound} and the choice of \(\delta_r\),
\(\|\mathcal K_\eta\|_{\mathcal L(\mathcal X_r)} \leq\frac12\).
Hence the Neumann series gives
\[
 \|\mathcal R_\eta\|_{\mathcal L(\mathcal X_r)}\leq2.
\]
Since \(\eta\mapsto\mathcal K_\eta\) is \(C^3\) in operator norm, so
is \(\eta\mapsto\mathcal R_\eta\). Differentiating
\((I-\mathcal K_\eta)\mathcal R_\eta=I\) gives
\[
 D_\eta\mathcal R_\eta[h]
 = \mathcal R_\eta
 \bigl(D_\eta\mathcal K_\eta[h]\bigr) \mathcal R_\eta.
\]
Repeated differentiation shows that \(D_\eta^k\mathcal R_\eta\), \(k\leq3\), is a finite sum of terms alternating factors \(\mathcal R_\eta\) with derivatives
\(D_\eta^m\mathcal K_\eta\), where the derivative directions are partitioned among the latter factors. 
Hence \eqref{eq:K-shape-bound}
gives
\[
 \left\|D_\eta^k\mathcal R_\eta[h_1,\ldots,h_k]\right\|_{\mathcal L(\mathcal X_r)}
 \leq C_{r,k}\prod_{\nu=1}^k\mathcal A_r(h_\nu).
\]
Since \(v_\eta=\mathcal R_\eta P[f]\), the Poisson estimate yields, for \(1\leq k\leq3\),
\begin{equation*}
 \left\| D_\eta^kv_\eta [h_1,\ldots,h_k] \right\|_{\mathcal X_r} \leq C_{r,k} \mathcal D_r(f) \prod_{\nu=1}^k \mathcal A_r(h_\nu).
\end{equation*}

The trace identity reads \(G(\eta)f=|D|f+\mathcal Z_\eta\mathcal R_\eta P[f]\).  Leibniz rule, \eqref{eq:K-shape-bound}, and the preceding resolvent bounds give
\[
 \left\| D_\eta^kG(\eta) [h_1,\ldots,h_k]f \right\|_{\dot H^r} \leq C_{r,k} \mathcal D_r(f) \prod_{\nu=1}^k \mathcal A_r(h_\nu), \qquad 1\leq k\leq3.
\]
The same operator formulas are meaningful directly for every
\(\eta\in\mathfrak A_r\) in the small ball and every
\(f\in\mathfrak D_r\); no density in the Lipschitz norm is required.
The coefficient estimates define \(Q_a,Q_b\), hence
\(\mathcal K_\eta\), on \(\mathcal X_r\), and the contraction formula
defines \(v_\eta\).  Substitution into the two first-order equations shows
in distributions that
\[
 \operatorname{div}_{x,z}(\mathsf A_\eta\nabla_{x,z}v_\eta)=0,
 \qquad v_\eta|_{z=0}=f.
\]
Moreover, \(\nabla v_\eta\in L^2(\R^3_-)\), so this is an energy
solution.  Composing with the inverse bi-Lipschitz flattening map gives an
energy harmonic extension in \(\Omega_\eta\), with boundary trace \(f\).
It is unique: the difference
\(u\in\dot H^1_0(\R^3_-)\) of two energy solutions satisfies
\[
 \int_{\R^3_-}\mathsf A_\eta\nabla u\cdot\nabla u=0.
\]
Ellipticity and the zero trace imply \(u=0\).  
Moreover,
\[
 (\mathsf A_\eta\nabla v_\eta)_3
 =\partial_zv_\eta-Q_a(\eta,v_\eta)=|D|v_\eta+w_\eta
 \longrightarrow |D|f+w_\eta(0)
\]
in boundary distributions.  Here the zero-boundary Duhamel term in
\(v_\eta-P[f]\) tends to zero, while the boundary estimate in
Lemma~\ref{lem:flat-mixed-elliptic} supplies the trace of \(w_\eta\).
Testing the weak equation against an extension \(\Psi\) of \(\psi\in C_c^\infty(\R^2)\) gives
\[
 \int_{\R^3_-}\mathsf A_\eta\nabla v_\eta\cdot\nabla\Psi =\langle |D|f+w_\eta(0),\psi\rangle.
\]
Thus \(|D|f+\mathcal Z_\eta v_\eta\) is the variational
Dirichlet--Neumann trace. This identifies the completed-space construction with the classical operator for smooth data and completes the proof of
\eqref{eq:finite-order-DN}.
\end{proof}

\begin{lemma}[Compatibility of the quadratic extension]
\label{lem:quadratic-compatibility}
Let \(3/2<r<2\), and let \(h_1,h_2\) satisfy
\[
 \mathfrak N_r(h_\nu)<\infty, \qquad \nu=1,2.
\]
Then the homogeneous \(\dot H^r\)-class \(\mathcal B(h_1,h_2)\) constructed in
Lemma~\ref{lem:quadratic-cancellation-revised} has the canonical tempered-distribution representative
\begin{equation}\label{eq:quadratic-compatibility}
 \mathcal B(h_1,h_2)
 = \frac12 \left( D_\eta G(0)[h_1](\Delta h_2)
  + D_\eta G(0)[h_2](\Delta h_1)
 \right).
\end{equation}
In particular,
\begin{equation}
\label{eq:quadratic-diagonal-compatibility}
 \mathcal B(h,h) = D_\eta G(0)[h](\Delta h)
\end{equation}
as an identity in \(\mathcal S'(\R^2)\).
\end{lemma}

\begin{proof}
We first verify that all the terms in
\eqref{eq:quadratic-compatibility} are well-defined. By the definition
of \(\mathfrak N_r\),
\[
 \mathcal A_r(h_\nu)
 \lesssim_r \mathfrak N_r(h_\nu).
\]
Moreover,
\[
 \mathcal D_r(\Delta h_\nu)
 = \|h_\nu\|_{\dot H^{5/2}} + \|h_\nu\|_{\dot H^3} + \|h_\nu\|_{\dot H^{r+3}} \lesssim_r \mathfrak N_r(h_\nu),
\]
where the \(\dot H^3\)-norm is obtained by homogeneous
interpolation between \(\dot H^{5/2}\) and \(\dot H^{r+3}\).
Thus Proposition~\ref{prop:finite-order-DN} defines
\[
 D_\eta G(0)[h_1](\Delta h_2),
 \qquad D_\eta G(0)[h_2](\Delta h_1)
\]
as elements of \(\dot H^r\) and as variational
Dirichlet--Neumann distributions.

We next identify their Littlewood--Paley blocks. Recall that
\[
 G(\eta)f = |D|f + \mathcal Z_\eta
 (I-\mathcal K_\eta)^{-1}P[f],
 \qquad P[f](z)=e^{z|D|}f.
\]
Since $\mathcal K_0=0$, and 
 $\mathcal Z_0=0$, differentiation in operator norm at the flat graph gives
\begin{equation} \label{eq:DG0-completed-formula}
 D_\eta G(0)[h]f = D_\eta\mathcal Z_0[h]P[f].
\end{equation}

The linearizations of \(Q_a\) and \(Q_b\) at the flat graph are
\begin{align}\label{eq:Qa-linearization-flat}
 D_\eta Q_a(0,V)[h]
 &= \nabla H_h\cdot\nabla V + \partial_zH_h\,\partial_zV,\\
\label{eq:Qb-linearization-flat}
 D_\eta Q_b(0,V)[h]
 &= |D|^{-1}\operatorname{div}
 \left( \nabla H_h\,\partial_zV  -  \partial_zH_h\,\nabla V
 \right).
\end{align}
Consequently, \eqref{eq:DG0-completed-formula} becomes
\begin{align}
\label{eq:DG0-Duhamel-formula}
 D_\eta G(0)[h]f = \int_{-\infty}^0 e^{z|D|}|D| D_\eta(Q_b-Q_a)(0,P[f])[h](z)\,dz.
\end{align}

We calculate the symbol of this expression. Let
\[
 \zeta=\alpha+\beta, \qquad Z:=|\zeta|, \qquad B:=|\beta|, \qquad \vartheta(z,\alpha) := \widehat\chi((-z)\alpha).
\]
Then
\[
 \widehat{H_h}(\alpha,z) = \vartheta(z,\alpha)\widehat h(\alpha), \qquad
 \widehat{P[f]}(\beta,z) = e^{zB}\widehat f(\beta).
\]
Substitution into
\eqref{eq:Qa-linearization-flat}--\eqref{eq:Qb-linearization-flat}
gives
\begin{align*}
 \widehat{D_\eta Q_a(0,P[f])[h]}(\zeta,z)
 &= \int_{\R^2} e^{zB} \left(
  -\alpha\cdot\beta\,\vartheta  +  B\,\partial_z\vartheta \right)
 \widehat h(\alpha)\widehat f(\beta)\,d\beta,\\
 Z\widehat{D_\eta Q_b(0,P[f])[h]}(\zeta,z)
 &= \int_{\R^2} e^{zB} \left( -B\zeta\cdot\alpha\,\vartheta + \zeta\cdot\beta\,\partial_z\vartheta \right)
 \widehat h(\alpha)\widehat f(\beta)\,d\beta.
\end{align*}
Using \(\alpha=\zeta-\beta\), we have
\[
 -B\zeta\cdot\alpha + Z\alpha\cdot\beta = (Z+B) \bigl(  \zeta\cdot\beta-ZB \bigr).
\]
It follows that
\begin{equation} \label{eq:DG0-total-z-derivative}
\begin{aligned}
&e^{zZ}Z\,
\widehat{D_\eta(Q_b-Q_a)(0,P[f])[h]}(\zeta,z)
 \\
 = \int_{\R^2}& \bigl(  \zeta\cdot\beta-ZB \bigr) \partial_z \left[ e^{z(Z+B)} \vartheta(z,\alpha) \right]
 \widehat h(\alpha)\widehat f(\beta)\,d\beta.
\end{aligned}
\end{equation}

For inputs with finitely many nonzero Littlewood--Paley blocks,
we may integrate \eqref{eq:DG0-total-z-derivative} over
\((-\infty,0)\). Since
\[
\vartheta(0,\alpha)=\widehat\chi(0)=1,
\]
and the lower boundary term vanishes, we obtain
\begin{equation}
\label{eq:flat-shape-symbol-completed}
 \widehat{D_\eta G(0)[h]f}(\zeta)
 = \int_{\R^2}
 \left(  \zeta\cdot\beta-|\zeta||\beta|
 \right) \widehat h(\zeta-\beta)
 \widehat f(\beta)\,d\beta.
\end{equation}
Equivalently,
\[
 D_\eta G(0)[h]f
 = -|D|(h|D|f) - \operatorname{div}(h\nabla f).
\]

Taking \(f=\Delta h_2\) in
\eqref{eq:flat-shape-symbol-completed} introduces the factor
\(-|\beta|^2\). Hence the bilinear symbol of
\(D_\eta G(0)[h_1](\Delta h_2)\) is
\[
 \left( |\alpha+\beta||\beta| -
  (\alpha+\beta)\cdot\beta \right)|\beta|^2,
\]
which is exactly the symbol defining \(\mathcal B_0(h_1,h_2)\). Symmetrizing in \(h_1,h_2\) gives
\eqref{eq:quadratic-compatibility} for inputs with finitely many
nonzero Littlewood--Paley blocks.

For general \(h_1,h_2\), we apply the preceding computation separately to every localized paraproduct and resonant interaction. 
In the low--high terms, the low-frequency blocks are summed before taking the \(L^\infty\)-estimate, exactly as in the proof of Lemma~\ref{lem:quadratic-cancellation-revised}. 
The estimates in that lemma show that the resulting output blocks converge in \(\dot H^r\) and coincide with the blocks defining the homogeneous class \(\mathcal B(h_1,h_2)\).

On the other hand, the completed Dirichlet--Neumann construction defines the right-hand side of \eqref{eq:quadratic-compatibility} as an actual tempered distribution. 
We take this distribution as the canonical representative of the corresponding homogeneous class. 
This proves
\eqref{eq:quadratic-compatibility} in \(\mathcal S'(\R^2)\). 
Setting \(h_1=h_2=h\) proves \eqref{eq:quadratic-diagonal-compatibility}.
\end{proof}

\begin{remark}
\label{rem:why-low-anchors}
The two low-regularity norms in \(\mathcal D_r(f)\) are essential.
Indeed, for the flat Poisson extension
\[
 P[f](z)=e^{z|D|}f,
\]
one has
\[
 \|\nabla_{x,z}P[f]\|_{L^2_{x,z}} \simeq \|f\|_{\dot H^{1/2}}, \qquad \|\nabla_{x,z}P[f]\|_{L^\infty_zL^2_x} \lesssim \|f\|_{\dot H^1}.
\]
Thus the \(\dot H^{1/2}\)-norm controls the bulk energy.
The \(\dot H^1\)-norm controls the low-frequency part of the
\(L^\infty_zL^2_x\)-norm.

The positive-frequency component of the \(\mathcal X_r\)-norm, together with the preceding \(L^\infty_zL^2_x\) control, yields
\[
 \|\nabla_{x,z}V\|_{L^\infty_{x,z}} \lesssim_r \|V\|_{\mathcal X_r},
\]
as stated in \eqref{eq:X-to-Linfty}.
The same conclusion for the flattened harmonic extension follows from the fixed-point estimate in \(\mathcal X_r\).
In particular, the argument never requires an \(L^\infty\)-to-\(L^\infty\) estimate for a singular integral.
\end{remark}

We record two consequences of Proposition~\ref{prop:finite-order-DN} in
the form needed for the nonlinear expansion.
\begin{lemma}[Cone-adapted DN estimate]
\label{lem:DN-one-input-revised}
Let \(3/2<r<2\), and let \(f\) satisfy \(\mathcal D_r(f)<\infty\).
If \(\mathcal A_r(\eta)\leq\delta_r\), then
\begin{equation*}
 \|G(\eta)f\|_{\dot H^r} \leq C_r\mathcal D_r(f).
\end{equation*}
Moreover, if \(\mathcal A_r(\eta_i)\leq\delta_r\) for \(i=1,2\), then
\begin{equation*}
 \|(G(\eta_1)-G(\eta_2))f\|_{\dot H^r} \leq C_r \mathcal A_r(\eta_1-\eta_2) \mathcal D_r(f).
\end{equation*}
\end{lemma}

\begin{proof}
The first estimate is the case \(k=0\) of
Proposition~\ref{prop:finite-order-DN}.  For
\(\eta_\theta=\eta_2+\theta(\eta_1-\eta_2)\), convexity of the norm keeps
the whole segment in the same ball, and
\[
 (G(\eta_1)-G(\eta_2))f
 =\int_0^1D_\eta G(\eta_\theta)[\eta_1-\eta_2]f\,d\theta.
\]
The case \(k=1\) of the proposition proves the difference estimate.
\end{proof}

For the Taylor remainder of \(G(h)\Delta h\), we also need estimates for
the second and third shape derivatives with conical inputs.

\begin{lemma}[Cone-specialized second and third variations]
\label{lem:DN-second-shape-revised}
Let \(3/2<r<2\), and suppose that
\(\mathcal A_r(\eta)\leq\delta_r\).
Then, for \(k=2,3\) and \(h_1,\ldots,h_{k+1}\) satisfying
\(\mathfrak N_r(h_\nu)<\infty\) for \(1\leq\nu\leq k+1\),
\begin{equation}
\label{eq:DN-k-shape-cone}
 \left\| D_\eta^kG(\eta) [h_1,\ldots,h_k]\Delta h_{k+1} \right\|_{\dot H^r} \leq C_r \prod_{\nu=1}^{k+1}\mathfrak N_r(h_\nu).
\end{equation}
Moreover, if \(\mathcal A_r(\eta_i)\leq\delta_r\) for \(i=1,2\) and \(\mathfrak N_r(\eta_1-\eta_2)<\infty\), then
\begin{align}
\label{eq:DN-second-shape-background-difference}
 \left\| \bigl( D_\eta^2G(\eta_1)-D_\eta^2G(\eta_2) \bigr) [h_1,h_2]\Delta h_3 \right\|_{\dot H^r}\leq C_r \mathfrak N_r(\eta_1-\eta_2) \prod_{\nu=1}^3\mathfrak N_r(h_\nu).
\end{align}
\end{lemma}

\begin{proof}
For every admissible \(h\), interpolation and the bound
\(\mathcal A_r(h)\lesssim_r\mathfrak N_r(h)\) give
\begin{equation} \label{eq:D-Delta-h-control}
 \mathcal D_r(\Delta h) = \|\Delta h\|_{\dot H^{1/2}} +\|\Delta h\|_{\dot H^1} +\|\Delta h\|_{\dot H^{r+1}}\leq C_r\mathfrak N_r(h). 
\end{equation}
Proposition~\ref{prop:finite-order-DN} with
\(f=\Delta h_{k+1}\), followed by these two bounds, proves
\eqref{eq:DN-k-shape-cone}.

For \(\eta_\theta=\eta_2+\theta(\eta_1-\eta_2)\),
\begin{align*}
 \bigl(D_\eta^2G(\eta_1)-D_\eta^2G(\eta_2)\bigr)[h_1,h_2]\Delta h_3
 =\int_0^1D_\eta^3G(\eta_\theta)
 [\eta_1-\eta_2,h_1,h_2]\Delta h_3\,d\theta.
\end{align*}
The case \(k=3\) of the proposition and
\eqref{eq:D-Delta-h-control} prove
\eqref{eq:DN-second-shape-background-difference}.
\end{proof}

Taylor's formula and the preceding variation bounds give the required
cubic remainder.

\begin{proposition}[Cone-adapted second-order DN remainder]
\label{prop:DN-second-remainder-revised}
Define the second-order DN remainder
\[
 \mathcal R_{\mathrm{DN},\geq2}(h) := \bigl( G(h)-|D|-D_\eta G(0)[h] \bigr)\Delta h.
\]
There exists \(\delta_r>0\) such that, if
\(\mathfrak N_r(h)\leq\delta_r\), then
\begin{equation*}
 \|\mathcal R_{\mathrm{DN},\geq2}(h)\|_{\dot H^r} \leq C_r\mathfrak N_r(h)^3.
\end{equation*}
Moreover, if \(\mathfrak N_r(h_i)\leq\delta_r\) for \(i=1,2\), then
\begin{align*}
 \left\| \mathcal R_{\mathrm{DN},\geq2}(h_1) - \mathcal R_{\mathrm{DN},\geq2}(h_2) \right\|_{\dot H^r}\leq C_r \bigl( \mathfrak N_r(h_1)+\mathfrak N_r(h_2) \bigr)^2 \mathfrak N_r(h_1-h_2).
\end{align*}
\end{proposition}

\begin{proof}
Taylor's formula with integral remainder gives
\begin{equation*}
 \mathcal R_{\mathrm{DN},\geq2}(h) = \int_0^1 (1-\theta) D_\eta^2G(\theta h)[h,h]\Delta h \,d\theta.
\end{equation*}
Since
\[
 \mathcal A_r(\theta h) \leq \theta\mathcal A_r(h) \leq
 \mathfrak N_r(h) \leq \delta_r,
 \qquad 0\leq\theta\leq1,
\]
the segment \(\{\theta h:0\leq\theta\leq1\}\) remains in the
small-background ball. Applying
\eqref{eq:DN-k-shape-cone} with \(k=2\) gives, uniformly in
\(\theta\),
\[
 \bigl\|
  D_\eta^2G(\theta h)[h,h]\Delta h
 \bigr\|_{\dot H^r} \leq C_r\mathfrak N_r(h)^3.
\]
Integration in \(\theta\) proves the cubic bound.

Set
\[
 \delta h:=h_1-h_2, \qquad N_i:=\mathfrak N_r(h_i), \qquad N_\delta:=\mathfrak N_r(\delta h).
\]
For fixed \(\theta\), write \(A_i(\theta):=D_\eta^2G(\theta h_i)\).
The difference of the two integrands is
\begin{align*}
 &A_1(\theta)[h_1,h_1]\Delta h_1 - A_2(\theta)[h_2,h_2]\Delta h_2\\
 =& A_1(\theta)[\delta h,h_1]\Delta h_1 + A_1(\theta)[h_2,\delta h]\Delta h_1+ A_1(\theta)[h_2,h_2]\Delta\delta h + \bigl(A_1(\theta)-A_2(\theta)\bigr) [h_2,h_2]\Delta h_2.
\end{align*}
The first three terms are controlled by \eqref{eq:DN-k-shape-cone}; the last is controlled by \eqref{eq:DN-second-shape-background-difference} and
is at most \(C_r\theta N_\delta N_2^3\).  
Taking \(\delta_r\) so that \(N_1+N_2\leq1\), all four terms are bounded by
\(C_r(N_1+N_2)^2N_\delta\).  
Integration in \(\theta\) proves the difference estimate.
\end{proof}

\section{Cubic remainders and the fixed-point construction}
\label{s:capillary-fixed-point}

Having isolated the favorable quadratic cancellation and established the required Dirichlet--Neumann estimates, it remains to control the cubic terms and close the construction by a contraction argument.

\subsection{The cubic capillary remainder}

We write the mean curvature operator as
\begin{equation*}
 \kappa(h) = -\Delta h+\mathcal Q_\kappa(h), \qquad
 \mathcal Q_\kappa(h) := \operatorname{div} \left[ \left( 1-\frac1{\sqrt{1+|\nabla h|^2}} \right)\nabla h \right].
\end{equation*}

We begin with a standard homogeneous composition estimate tailored to
cubic nonlinearities.

\begin{lemma}[Homogeneous cubic composition]
\label{lem:homogeneous-cubic-composition}
Let \(s>0\), and let \(\Phi\) be smooth on a neighborhood of the closed
unit ball in a finite-dimensional vector space, with
\begin{equation} \label{PhiZero}
 \Phi(0)=D\Phi(0)=D^2\Phi(0)=0.
\end{equation}
If \(p,p_1,p_2\in L^\infty\cap\dot H^s(\R^2)\) take values in that closed
unit ball, then
\begin{align*}
 \|\Phi(p)\|_{\dot H^s} &\leq C_{\Phi,s}\|p\|_{L^\infty}^2\|p\|_{\dot H^s},\\
 \|\Phi(p_1)-\Phi(p_2)\|_{\dot H^s} &\leq C_{\Phi,s}\Bigl[ (\|p_1\|_{L^\infty}+\|p_2\|_{L^\infty})^2 \|p_1-p_2\|_{\dot H^s}\\
 &\hspace{22mm}
 +(\|p_1\|_{L^\infty}+\|p_2\|_{L^\infty})
 (\|p_1\|_{\dot H^s}+\|p_2\|_{\dot H^s})
 \|p_1-p_2\|_{L^\infty}\Bigr].
\end{align*}
\end{lemma}

\begin{proof}
The homogeneous Bony decomposition gives, for \(s>0\),
\[
 \|abc\|_{\dot H^s}
 \lesssim_s
 \|a\|_{\dot H^s}\|b\|_{L^\infty}\|c\|_{L^\infty}
 +\|a\|_{L^\infty}\|b\|_{\dot H^s}\|c\|_{L^\infty}
 +\|a\|_{L^\infty}\|b\|_{L^\infty}\|c\|_{\dot H^s}.
\]
The same decomposition gives the homogeneous composition estimate
\(\|A(p)\|_{\dot H^s}\lesssim_{A,s,\|p\|_\infty}\|p\|_{\dot H^s}\)
for a smooth function \(A\).

By \eqref{PhiZero}, Taylor's formula gives
\[
 \Phi(p)=\frac12\int_0^1(1-\theta)^2 D^3\Phi(\theta p)[p,p,p]\,d\theta.
\]
Applying the preceding product and composition estimates, and using
\(\|p\|_{L^\infty}\leq1\), proves the first inequality. 
For the difference, use
\[
 \Phi(p_1)-\Phi(p_2) =\int_0^1D\Phi(p_2+\theta(p_1-p_2))[p_1-p_2]\,d\theta.
\]
Since \(D\Phi(q)\) vanishes quadratically at \(q=0\), the same tame product estimate, with the \(\dot H^s\)-derivative placed either on \(p_1-p_2\) or on one copy of \(q\), gives the second inequality.
\end{proof}

Applying this estimate to the curvature nonlinearity gives the bounds needed below.
\begin{lemma}[Cubic curvature estimate]
\label{lem:curvature-revised}
Let \(3/2<r<2\).
There exists \(\delta_r>0\) such that, if \(\mathfrak N_r(h)\leq\delta_r\), then
\begin{equation*}
\begin{aligned}
 \|\mathcal Q_\kappa(h)\|_{\dot H^{1/2}} + \|\mathcal Q_\kappa(h)\|_{\dot H^1} + \|\mathcal Q_\kappa(h)\|_{\dot H^{r+1}}+ \|\nabla\mathcal Q_\kappa(h)\|_{L^\infty} \leq C_r\mathfrak N_r(h)^3.
\end{aligned}
\end{equation*}
Moreover, if \(\mathfrak N_r(h_i)\leq\delta_r\) for \(i=1,2\), then
\begin{equation*}
\begin{aligned}
 &\|\mathcal Q_\kappa(h_1)-\mathcal Q_\kappa(h_2)\|_{\dot H^{1/2}} + \|\mathcal Q_\kappa(h_1)-\mathcal Q_\kappa(h_2)\|_{\dot H^1}+ \|\mathcal Q_\kappa(h_1)-\mathcal Q_\kappa(h_2)\|_{\dot H^{r+1}}\\
 &\quad+ \|\nabla(\mathcal Q_\kappa(h_1)-\mathcal Q_\kappa(h_2))\|_{L^\infty}\leq C_r \bigl( \mathfrak N_r(h_1)+\mathfrak N_r(h_2) \bigr)^2 \mathfrak N_r(h_1-h_2).
\end{aligned}
\end{equation*}
\end{lemma}

\begin{proof}
If we write \(F(p):=(1-(1+|p|^2)^{-1/2})p\), then
\[
 \mathcal Q_\kappa(h)=\operatorname{div}F(\nabla h).
\]
Since \(F\) vanishes to third order at the origin, for \(|p|\leq1\) one has
\begin{equation*}
 |F(p)|\lesssim|p|^3, \qquad |DF(p)|\lesssim|p|^2, \qquad |D^2F(p)|\lesssim|p|, \qquad |D^3F(p)|\lesssim1.
\end{equation*}
Lemma~\ref{lem:homogeneous-cubic-composition} therefore gives, for
\(s>0\) and \(\|p_i\|_{L^\infty}\leq1\),
\begin{align}
\label{eq:cubic-Moser-combined}
 \|F(p)\|_{\dot H^s} &\leq C_s\|p\|_{L^\infty}^2\|p\|_{\dot H^s},
\notag\\[-1mm]
 \|F(p_1)-F(p_2)\|_{\dot H^s}
 \lesssim_s{}& (\|p_1\|_{L^\infty}+\|p_2\|_{L^\infty})^2
 \|p_1-p_2\|_{\dot H^s}\\
 &+(\|p_1\|_{L^\infty}+\|p_2\|_{L^\infty})
 (\|p_1\|_{\dot H^s}+\|p_2\|_{\dot H^s})
 \|p_1-p_2\|_{L^\infty}.
\notag
\end{align}

Bernstein's inequality, split at frequency one, gives
\begin{equation}
\label{eq:D3h-Linfty-revised}
 \|D^3h\|_{L^\infty} \lesssim_r \|h\|_{\dot H^{r+2}} + \|h\|_{\dot H^{r+3}}.
\end{equation}
Indeed, the low- and high-frequency summation weights are respectively
\(2^{(2-r)j}\) and \(2^{(1-r)j}\).
Interpolation also gives
\[
 \|h\|_{\dot H^3} \lesssim_r \|h\|_{\dot H^{5/2}} + \|h\|_{\dot H^{r+3}} \lesssim_r \mathfrak N_r(h).
\]
Applying \eqref{eq:cubic-Moser-combined} to \(p=\nabla h\) with
\(s=3/2,2,r+2\), and then taking one divergence, yields
\[
 \|\mathcal Q_\kappa(h)\|_{\dot H^{1/2}} + \|\mathcal Q_\kappa(h)\|_{\dot H^1} + \|\mathcal Q_\kappa(h)\|_{\dot H^{r+1}} \lesssim_r \mathfrak N_r(h)^3.
\]

Differentiating once more gives schematically
\[
 \nabla\mathcal Q_\kappa(h) = D^2F(\nabla h)[D^2h,D^2h] + DF(\nabla h)D^3h.
\]
Hence, by \eqref{eq:D3h-Linfty-revised},
\[
 \|\nabla\mathcal Q_\kappa(h)\|_{L^\infty} \lesssim \|\nabla h\|_{L^\infty} \|D^2h\|_{L^\infty}^2+ \|\nabla h\|_{L^\infty}^2 \|D^3h\|_{L^\infty}.
\]
This proves the first assertion.

Set \(p_i:=\nabla h_i\), \(\delta h:=h_1-h_2\), \(\delta p:=\nabla\delta h\), \(M:=\mathfrak N_r(h_1)+\mathfrak N_r(h_2)\), and \(M_\delta:=\mathfrak N_r(\delta h)\).
The difference estimate in \eqref{eq:cubic-Moser-combined}, at the same
three indices, gives
\[
 \|F(p_1)-F(p_2)\|_{\dot H^{3/2}} + \|F(p_1)-F(p_2)\|_{\dot H^2}+ \|F(p_1)-F(p_2)\|_{\dot H^{r+2}} \lesssim_r M^2M_\delta.
\]
After one divergence this proves
\[
 \|\mathcal Q_\kappa(h_1)-\mathcal Q_\kappa(h_2)\|_{\dot H^{1/2}} + \|\mathcal Q_\kappa(h_1)-\mathcal Q_\kappa(h_2)\|_{\dot H^1}+ \|\mathcal Q_\kappa(h_1)-\mathcal Q_\kappa(h_2)\|_{\dot H^{r+1}} \lesssim_r M^2M_\delta.
\]

Subtracting the two identities for
\(\nabla\mathcal Q_\kappa\) produces terms of the following forms:
\begin{align*}
 &[D^2F(p_1)-D^2F(p_2)][D^2h_1,D^2h_1],\qquad D^2F(p_2)[D^2\delta h,D^2h_1], \\
 &D^2F(p_2)[D^2h_2,D^2\delta h],\qquad [DF(p_1)-DF(p_2)]D^3h_1, \qquad DF(p_2)D^3\delta h.
\end{align*}
The pointwise bounds above and the mean-value theorem give
\[
 |D^2F(p_1)-D^2F(p_2)|\lesssim|\delta p|,
 \qquad
 |DF(p_1)-DF(p_2)|\lesssim(|p_1|+|p_2|)|\delta p|.
\]
Together with \eqref{eq:D3h-Linfty-revised}, each displayed term is
bounded in \(L^\infty\) by \(C_rM^2M_\delta\).  Thus
\[
 \|\nabla(\mathcal Q_\kappa(h_1)-\mathcal Q_\kappa(h_2))\|_{L^\infty} \lesssim_r M^2M_\delta,
\]
which proves the gradient difference estimate.
\end{proof}

We can now combine the curvature estimate with the second-order
Dirichlet--Neumann expansion.

\begin{proposition}[Cubic nonlinear remainder] \label{prop:capillary-remainder-revised}
Let \(3/2<r<2\).
There exists \(\delta_r>0\) such that, if
\(\mathfrak N_r(h)\leq\delta_r\), then
\begin{equation*}
 \mathscr F(h) = \mathcal B(h,h)+\mathscr R_{\geq3}(h), \qquad
 \|\mathscr R_{\geq3}(h)\|_{\dot H^r} \leq C_r\mathfrak N_r(h)^3.
\end{equation*}
Moreover, if \(\mathfrak N_r(h_i)\leq\delta_r\) for \(i=1,2\), then
\begin{equation*}
\begin{aligned}
 \|\mathscr R_{\geq3}(h_1)-\mathscr R_{\geq3}(h_2)\|_{\dot H^r}\leq C_r \bigl( \mathfrak N_r(h_1)+\mathfrak N_r(h_2) \bigr)^2 \mathfrak N_r(h_1-h_2).
\end{aligned}
\end{equation*}
\end{proposition}

\begin{proof}
Choose \(\delta_r\) small enough that it does not exceed the smallness
thresholds in Proposition~\ref{prop:DN-second-remainder-revised},
Lemma~\ref{lem:DN-one-input-revised}, and
Lemma~\ref{lem:curvature-revised}.
Since \(\mathcal A_r(h)\leq\mathfrak N_r(h)\), the assumption \(\mathfrak N_r(h)\leq\delta_r\) ensures that all these estimates apply.

Using
\(\kappa(h)=-\Delta h+\mathcal Q_\kappa(h),
\)
expanding \(G(h)\) to first order, and using
\[
 |D|\Delta=-|D|^3,
 \qquad
 D_\eta G(0)[h](\Delta h)=\mathcal B(h,h)
\]
from Lemma~\ref{lem:quadratic-compatibility}, we obtain
\begin{equation}
\label{eq:Rge3-definition-revised}
\mathscr F(h) = \mathcal B(h,h)+\mathscr R_{\geq3}(h), \qquad 
 \mathscr R_{\geq3}(h) := \mathcal R_{\mathrm{DN},\geq2}(h) - G(h)\mathcal Q_\kappa(h).
\end{equation}

Proposition~\ref{prop:DN-second-remainder-revised} controls the first term cubically, while the one-input DN estimate and Lemma~\ref{lem:curvature-revised} give
\[
 \|G(h)\mathcal Q_\kappa(h)\|_{\dot H^r}
 \lesssim_r\mathcal D_r(\mathcal Q_\kappa(h))
 \lesssim_r\mathfrak N_r(h)^3.
\]
This proves the first estimate.

Set
\[
 \delta h:=h_1-h_2, \qquad M_i:=\mathfrak N_r(h_i), \qquad M_\delta:=\mathfrak N_r(\delta h), \qquad Q_i:=\mathcal Q_\kappa(h_i).
\]
From \eqref{eq:Rge3-definition-revised},
\begin{align*}
 \mathscr R_{\geq3}(h_1)-\mathscr R_{\geq3}(h_2) = \mathcal R_{\mathrm{DN},\geq2}(h_1) - \mathcal R_{\mathrm{DN},\geq2}(h_2)- G(h_1)(Q_1-Q_2) - \bigl(G(h_1)-G(h_2)\bigr)Q_2.
\end{align*}
The first three terms are bounded by \(C_r(M_1+M_2)^2M_\delta\) using the Dirichlet--Neumann remainder and curvature difference estimates. 
The Dirichlet--Neumann difference estimate bounds the last term by \(C_rM_\delta M_2^3\), which is no larger after taking \(M_1+M_2\leq1\).  
This proves the difference estimate.
\end{proof}

The preceding proposition applies uniformly to the conical background and corrections in \(X_r\).

\begin{corollary}[Expansion around the conical background]
\label{cor:cone-expansion-revised}
There exists \(\delta_{r,\sigma/\mu}>0\) such that, if
\(|\varepsilon|+\|v\|_{X_r}\leq\delta_{r,\sigma/\mu}\), then
\begin{equation*}
 \|\mathscr R_{\geq3}(\varepsilon u_*+v)\|_{\dot H^r} \leq C_{r,\sigma/\mu} \bigl( |\varepsilon|+\|v\|_{X_r} \bigr)^3.
\end{equation*}
Moreover, if \(|\varepsilon|+\|v_i\|_{X_r}\leq\delta_{r,\sigma/\mu}\) for \(i=1,2\), then
\begin{equation*}
\|\mathscr R_{\geq3}(\varepsilon u_*+v_1) - \mathscr R_{\geq3}(\varepsilon u_*+v_2)\|_{\dot H^r}\leq C_{r,\sigma/\mu} \bigl( |\varepsilon| +\|v_1\|_{X_r} +\|v_2\|_{X_r} \bigr)^2 \|v_1-v_2\|_{X_r}.
\end{equation*}
\end{corollary}

\begin{proof}
Lemma~\ref{lem:Xr-cone-embedding-revised} gives
\begin{align*}
 \mathfrak N_r(\varepsilon u_*+v)\lesssim_{r,\sigma/\mu}
 |\varepsilon|+\|v\|_{X_r},\qquad \mathfrak N_r(v_1-v_2)\lesssim_r\|v_1-v_2\|_{X_r}.
\end{align*}
Choose the threshold so that these bounds lie in the small ball as in Proposition~\ref{prop:capillary-remainder-revised}.  
Its two estimates, applied to \(h_i=\varepsilon u_*+v_i\), give both conclusions.
\end{proof}

\subsection{Fixed-point formulation}

We formulate the contraction directly for the correction \(v\in X_r\).
This avoids introducing an auxiliary inverse-Laplacian variable and then recovering the affine component of \(v\).

Lemma~\ref{lem:quadratic-compatibility} identifies the homogeneous
class \(\mathcal B(h,h)\) with the actual distribution
\(
 \mathcal B(h,h) = D_\eta G(0)[h](\Delta h).
\)
Moreover, Proposition~\ref{prop:capillary-remainder-revised} defines \(\mathscr R_{\geq3}(h)\) through the variational Dirichlet--Neumann construction, and hence also as an actual tempered distribution. 
Both quantities are unchanged when \(h\) is replaced by \(h+c\): vertical translation leaves the graph Dirichlet--Neumann operator unchanged in horizontal coordinates,
while all the remaining terms involve derivatives of \(h\).
Consequently,
\(
 \mathcal B(h,h)+\mathscr R_{\geq3}(h)
\)
is a well-defined tempered distribution, not merely a homogeneous
Sobolev class.
The observation following Proposition~\ref{prop:linear-inverse-Sobolev} shows that
\(\mathcal L^{-1}\) is well-defined on classes modulo constants and that its physical formula selects an actual representative when the datum is an actual distribution. 
Consequently, the map \eqref{eq:Phi-epsilon-definition-revised} below is well-defined on the quotient space \(X_r\), and the displayed formula selects the representative used in the profile equation.

For \(v\in X_r\), set \(h_v:=\varepsilon u_*+v\) and define
\begin{equation}
\label{eq:Phi-epsilon-definition-revised}
 \Phi_\varepsilon(v) := \frac{3\sigma}{\mu}\mathcal L^{-1} \left[ \mathcal B(h_v,h_v) + \mathscr R_{\geq3}(h_v) \right].
\end{equation}
Here \(\mathcal L^{-1}\) is the canonical normalized inverse from
Proposition~\ref{prop:linear-inverse-Sobolev}.

The estimates established above now yield the contraction.

\begin{proposition}[Fixed-point construction]
\label{prop:fixed-point-construction-revised}
There exist constants
\[
 \varepsilon_{r,\sigma/\mu}>0, \qquad K=K_{r,\sigma/\mu}>0,
\]
such that, for every \(\varepsilon\) satisfying
\(|\varepsilon|\leq\varepsilon_{r,\sigma/\mu}\),
the map \(\Phi_\varepsilon\) defined by \eqref{eq:Phi-epsilon-definition-revised} is well-defined on the closed ball
\[
 \mathbb B_\varepsilon^X := \left\{ v\in X_r: \|v\|_{X_r}\leq K\varepsilon^2 \right\}.
\]
It maps \(\mathbb B_\varepsilon^X\) into itself and is a contraction there.
Consequently, there exists a unique
\(v_\varepsilon\in\mathbb B_\varepsilon^X\) such that
\[
 v_\varepsilon=\Phi_\varepsilon(v_\varepsilon).
\]

If \(U_\varepsilon:=\varepsilon u_*+v_\varepsilon\), then \(U_\varepsilon\) solves the profile equation \eqref{AbstractEqn}, and
\begin{equation}
\label{eq:v-fixed-point-bound}
 \|v_\varepsilon\|_{X_r} \leq C_{r,\sigma/\mu}\varepsilon^2.
\end{equation}
\end{proposition}

\begin{proof}
Let \(v\in X_r\) and set \(h_v:=\varepsilon u_*+v\).
By Lemma~\ref{lem:Xr-cone-embedding-revised},
\[
 \mathfrak N_r(h_v)\lesssim_{r,\sigma/\mu}|\varepsilon|+\|v\|_{X_r}.
\]
The quadratic estimate, the cubic remainder estimate, and the linear
inverse give, for small arguments,
\begin{equation}
\label{eq:Phi-self-map-short}
 \|\Phi_\varepsilon(v)\|_{X_r} \leq C_0 \bigl( |\varepsilon|+\|v\|_{X_r} \bigr)^2.
\end{equation}

Let \(h_i:=\varepsilon u_*+v_i\) and \(\delta v:=v_1-v_2\).
The bilinear identity, Corollary~\ref{cor:cone-expansion-revised}, and the
linear inverse estimate yield
\begin{equation}
\label{eq:Phi-contraction-short}
 \|\Phi_\varepsilon(v_1) -\Phi_\varepsilon(v_2)\|_{X_r}\leq C_1 \bigl( |\varepsilon| +\|v_1\|_{X_r} +\|v_2\|_{X_r} \bigr) \|v_1-v_2\|_{X_r}
\end{equation}
after reducing the common smallness threshold.

Set \(K=4C_0\) and choose \(\varepsilon_{r,\sigma/\mu}>0\) so that
\begin{equation*}
\begin{aligned}
 K\varepsilon^2\leq|\varepsilon|,\quad
 C_{r,\sigma/\mu} \bigl( |\varepsilon|+K\varepsilon^2 \bigr) \leq\delta_{r,\sigma/\mu},\quad
 C_1 \bigl( |\varepsilon|+2K\varepsilon^2 \bigr) \leq\frac12,
\end{aligned}
\end{equation*}
for every \(|\varepsilon|\leq\varepsilon_{r,\sigma/\mu}\).

If \(v\in\mathbb B_\varepsilon^X\), then
\[
 |\varepsilon|+\|v\|_{X_r} \leq |\varepsilon|+K\varepsilon^2 \leq 2|\varepsilon|.
\]
Thus \eqref{eq:Phi-self-map-short} gives
\[
 \|\Phi_\varepsilon(v)\|_{X_r} \leq C_0(2|\varepsilon|)^2 = K\varepsilon^2.
\]
For \(v_1,v_2\) in the ball, \eqref{eq:Phi-contraction-short} gives
\[
 \|\Phi_\varepsilon(v_1) -\Phi_\varepsilon(v_2)\|_{X_r} \leq \frac12\|v_1-v_2\|_{X_r}.
\]
Since \(X_r\), modulo additive constants, is a Banach space, Banach's
fixed-point theorem gives a unique fixed-point class in
\(\mathbb B_\varepsilon^X\). The image of \(\Phi_\varepsilon\) lies in
the canonical normalized range of \(\mathcal L^{-1}\), so this class has
a unique representative \(v_\varepsilon\) in that range. In particular,
\(v_\varepsilon=\Phi_\varepsilon(v_\varepsilon)\) holds in
\(\mathcal S'(\R^2)\).

By the definition of \(\Phi_\varepsilon\) and
Proposition~\ref{prop:capillary-remainder-revised},
\[
 v_\varepsilon = \frac{3\sigma}{\mu}\mathcal L^{-1} \mathscr F(\varepsilon u_*+v_\varepsilon).
\]
Since \(\mathcal L^{-1}\) is a distributional right inverse and
\(\mathcal Lu_*=0\), \(U_\varepsilon=\varepsilon u_*+v_\varepsilon\)
satisfies
\[
 \mathcal LU_\varepsilon = \frac{3\sigma}{\mu}\mathscr F(U_\varepsilon).
\]
Finally, since \(v_\varepsilon\in\mathbb B_\varepsilon^X\),
\[
 \|v_\varepsilon\|_{X_r}\leq K\varepsilon^2,
\]
which proves \eqref{eq:v-fixed-point-bound}.
\end{proof}

The same fixed-point identity also determines the first nonlinear correction to the profile.

\begin{proposition}[Second-order expansion of the profile]
\label{prop:second-order-profile-expansion}
Let \(v_\varepsilon\) be the correction constructed in Proposition~\ref{prop:fixed-point-construction-revised}, and let \(v^{(2)}\) be the quadratic correction defined in Theorem~\ref{t:MainTheorem}.
Then \(v^{(2)}\in X_r\) and the second-order estimate stated in Theorem~\ref{t:MainTheorem} holds.
In particular,
\[
 \frac{v_\varepsilon}{\varepsilon^2} \longrightarrow v^{(2)}
 \qquad\text{in }X_r
 \quad\text{as }\varepsilon\to0,\quad \varepsilon\neq0.
\]
\end{proposition}

\begin{proof}
The quadratic estimate \eqref{eq:B-estimate-revised} and Lemma~\ref{lem:regularized-cone-revised} give
\[
 \|\mathcal B(u_*,u_*)\|_{\dot H^r} \leq C_r \|D^2u_*\|_{L^\infty} \|u_*\|_{\dot H^{r+2}} <\infty.
\]
Hence Proposition~\ref{prop:linear-inverse-Sobolev} implies
\[
 v^{(2)}\in X_r, \qquad \|v^{(2)}\|_{X_r}\leq C_{r,\sigma/\mu}.
\]
Set \(h_\varepsilon:=\varepsilon u_*+v_\varepsilon\).
Expanding the fixed-point equation and subtracting the quadratic term gives
\begin{equation*}
\begin{aligned}
 v_\varepsilon-\varepsilon^2v^{(2)} = \frac{3\sigma}{\mu}\mathcal L^{-1} \Big[ 2\varepsilon\mathcal B(u_*,v_\varepsilon) +\mathcal B(v_\varepsilon,v_\varepsilon) +\mathscr R_{\geq3}(h_\varepsilon) \Big].
\end{aligned}
\end{equation*}
The bilinear estimate, \eqref{eq:v-fixed-point-bound}, and
Corollary~\ref{cor:cone-expansion-revised} give
\begin{equation*}
 \|\mathcal B(u_*,v_\varepsilon)\|_{\dot H^r}\lesssim_{r,\sigma/\mu}\varepsilon^2, \quad
 \|\mathcal B(v_\varepsilon,v_\varepsilon)\|_{\dot H^r}\lesssim_{r,\sigma/\mu}\varepsilon^4,\quad
 \|\mathscr R_{\geq3}(h_\varepsilon)\|_{\dot H^r} \lesssim_{r,\sigma/\mu}|\varepsilon|^3.
\end{equation*}
The linear inverse estimate now proves
\(\|v_\varepsilon-\varepsilon^2v^{(2)}\|_{X_r}
\lesssim_{r,\sigma/\mu}|\varepsilon|^3\).
\end{proof}

We conclude by returning from the profile equation to the
time-dependent Muskat problem.
\begin{proof}[Proof of Theorem~\ref{t:MainTheorem}]
Proposition~\ref{prop:fixed-point-construction-revised} gives
\(U_\varepsilon=\varepsilon u_*+v_\varepsilon\) and
\[
 U_\varepsilon - x\cdot\nabla U_\varepsilon + \frac{3\sigma}{\mu}G(U_\varepsilon)\kappa(U_\varepsilon) = 0
\]
in \(\mathcal S'(\R^2)\).  Moreover,
\[
 \mathfrak N_r(U_\varepsilon) \leq C_{r,\sigma/\mu} \bigl( |\varepsilon|+\|v_\varepsilon\|_{X_r} \bigr),
\]
while \eqref{eq:D-Delta-h-control} and Lemma~\ref{lem:curvature-revised}
give \(\mathcal D_r(\kappa(U_\varepsilon))<\infty\).  Thus the energy
Dirichlet--Neumann construction of Proposition~\ref{prop:finite-order-DN}
applies.

Define
\[
 \eta_\varepsilon(t,x) := t^{1/3}U_\varepsilon(t^{-1/3}x), \qquad t>0.
\]
With the self-similar variable \(y=t^{-1/3}x\),
\[
 \partial_t\eta_\varepsilon(t,x) = \frac13t^{-2/3} \left( U_\varepsilon(y) - y\cdot\nabla U_\varepsilon(y) \right).
\]

Put \(s=t^{1/3}\).  The curvature scaling follows directly from
\(\eta_\varepsilon(t,x)=sU_\varepsilon(x/s)\).  To verify the weak
Dirichlet--Neumann scaling, let \(\phi_f\) be the energy harmonic extension
of \(f\) in \(\Omega_{U_\varepsilon}\), and define
\[
 \phi_{f,s}(x,y):=s^{-1}\phi_f(x/s,y/s).
\]
Then \(\phi_{f,s}\) is harmonic in \(\Omega_{sU_\varepsilon(\cdot/s)}\),
has boundary trace \(s^{-1}f(\cdot/s)\), and belongs to the energy class.
Moreover,
\(\nabla_{x,y}\phi_{f,s}(x,y)=s^{-2}(\nabla_{x,y}\phi_f)(x/s,y/s)\).
Since the graph slope is unchanged by this scaling, uniqueness of the
energy extension and the conormal-trace definition give
\[
\begin{aligned}
 \kappa(\eta_\varepsilon(t))(x) &= t^{-1/3}\kappa(U_\varepsilon)(y),\\
 G(\eta_\varepsilon(t)) \left[ t^{-1/3}f(t^{-1/3}\,\cdot\,) \right](x) &= t^{-2/3} \bigl(G(U_\varepsilon)f\bigr)(y).
\end{aligned}
\]
Taking \(f=\kappa(U_\varepsilon)\), the profile equation then gives
\[
 \partial_t\eta_\varepsilon = -\frac{\sigma}{\mu}G(\eta_\varepsilon)\kappa(\eta_\varepsilon), \qquad t>0.
\]

If \(p(t)\) is the energy harmonic extension of
\(\sigma\kappa(\eta_\varepsilon(t))\) and
\[
 u(t):=-\mu^{-1}\nabla_{x,y}p(t),
\]
then Darcy's law, incompressibility, and the dynamic condition hold by
construction, while the graph equation gives the kinematic condition.

The profile estimates follow from
Proposition~\ref{prop:fixed-point-construction-revised} and
Proposition~\ref{prop:second-order-profile-expansion}.

By Remark~\ref{rem:initial-profile},
\[
 t^{1/3}u_*(t^{-1/3}x) \longrightarrow |x|
\]
locally uniformly in \(x\).
For the normalized \(v_\varepsilon\), Morrey's estimate gives, with
\(\lambda=t^{-1/3}\),
\[
 t^{1/3}|v_\varepsilon(t^{-1/3}x)|
 \leq\lambda^{-1}|v_\varepsilon(0)|
 +C_r\lambda^{r-2}|x|^{r-1}\|v_\varepsilon\|_{\dot H^r}
 \longrightarrow0,
\]
locally uniformly as \(t\downarrow0\).
Hence \(\eta_\varepsilon(t,x)\to\varepsilon|x|\) locally uniformly.
For \(\varepsilon\neq0\), this trace is nonzero, and the solution is therefore nontrivial.
\end{proof}

\bibliographystyle{plain}
\bibliography{HWW}

\end{document}